\documentclass[11pt]{article}

\usepackage{epsfig}
\usepackage{graphicx}
\usepackage{amsbsy}
\usepackage{amsmath}
\usepackage{amsfonts}
\usepackage{amssymb}
\usepackage{textcomp}
\usepackage{hyperref}
\usepackage{aliascnt}

\newcommand{\mcm}[3]{\newcommand{#1}[#2]{{\ensuremath{#3}}}} 

\mcm{\tuple}{1}{\langle #1 \rangle}
\mcm{\name}{1}{\ulcorner #1 \urcorner}
\mcm{\Nbb}{0}{\mathbb{N}}
\mcm{\Zbb}{0}{\mathbb{Z}}
\mcm{\Rbb}{0}{\mathbb{R}}
\mcm{\Cbb}{0}{\mathbb{C}}
\mcm{\Qbb}{0}{\mathbb{Q}}
\mcm{\Bcal}{0}{\cal B}
\mcm{\Ccal}{0}{\cal C}
\mcm{\Dcal}{0}{\cal D}
\mcm{\Ecal}{0}{\cal E}
\mcm{\Fcal}{0}{\cal F}
\mcm{\Gcal}{0}{\cal G}
\mcm{\Hcal}{0}{\cal H}
\mcm{\Ical}{0}{\cal I}
\mcm{\Lcal}{0}{\cal L}
\mcm{\Mcal}{0}{\cal M}
\mcm{\Ncal}{0}{\cal N}
\mcm{\Ocal}{0}{{\cal O}}
\mcm{\Pcal}{0}{{\cal P}}
\mcm{\Scal}{0}{{\cal S}}
\mcm{\Tcal}{0}{{\cal T}}
\mcm{\Ucal}{0}{{\cal U}}
\mcm{\Vcal}{0}{{\cal V}}
\mcm{\Xcal}{0}{{\cal X}}
\mcm{\Ycal}{0}{{\cal Y}}
\mcm{\Mfrak}{0}{\mathfrak M}

\mcm{\restric}{0}{\upharpoonright}
\mcm{\upset}{0}{\uparrow}
\mcm{\onto}{0}{\twoheadrightarrow}
\mcm{\smallNbb}{0}{{\small \mathbb{N}}}
\DeclareMathOperator{\preop}{op}
\mcm{\op}{0}{^{\preop}}

\newcommand{\itum}{\item[$\bullet$]}
\newcommand{\se}{\subseteq}
{\begin{array}{c}
\setlength{\unitlength}{1em}}%
{\end{array}}

\usepackage{amsthm}

\newcommand{\theoremize}[2]{\newaliascnt{#1}{thm} \newtheorem{#1}[#1]{#2} \aliascntresetthe{#1}}

\theoremstyle{plain}
\newtheorem{thm}{Theorem}[section]
\theoremize{lem}{Lemma}
\theoremize{sublem}{Sublemma}
\theoremize{claim}{Claim}
\theoremize{rem}{Remark}
\theoremize{prop}{Proposition}
\theoremize{cor}{Corollary}
\theoremize{que}{Question}
\theoremize{oque}{Open Question}
\theoremize{con}{Conjecture}

\theoremstyle{definition}
\theoremize{dfn}{Definition}
\theoremize{eg}{Example}
\theoremize{exercise}{Exercise}
\theoremstyle{plain}

\usepackage{verbatim}
\usepackage[all]{xy}

\usepackage{subfig}

\title{\scshape Infinite graphic matroids \newline
 Part I}
\author{Nathan Bowler \and Johannes Carmesin \and Robin Christian}

\newcommand{\sm}{\setminus}

\DeclareMathOperator{\im}{Im}
\DeclareMathOperator{\Br}{Br}

\begin{document}
 
\maketitle
\begin{abstract}

An infinite matroid is \emph{graphic} if all of its finite minors are graphic
and the intersection of any circuit with any cocircuit is finite.
We show that a matroid is graphic if and only if
it can be represented by a graph-like topological space:
that is, a graph-like space in the sense of Thomassen and Vella. 
This extends Tutte's characterization of finite graphic matroids.

The representation we construct has many pleasant topological properties.
Working in the representing space, we prove that any circuit in a $3$-connected graphic matroid is countable.
\end{abstract}

\section{Introduction}

There is a rich theory describing and employing the relationship between finite graphic matroids and finite graphs. In this work, we attempt to extend this theory to infinite matroids \cite{matroid_axioms}. Instead of beginning with a class of infinite graphs, or compactifications thereof, and investigating the matroids that can be obtained from them (as in, for example, \cite{RD:HB:graphmatroids}), we start by specifying the class of matroids that we are interested in: A matroid is \emph{graphic} if, firstly, all its finite minors are graphic, and secondly, it is 
 \emph{tame} - that is, every circuit-cocircuit intersection is finite. 
We will explain the restriction to tame matroids below.
Very roughly, we shall show that every graphic matroid can be represented by a topological space that looks like a graph.

Before going into more detail, we will discuss some advantages of this approach. 

\begin{enumerate}
\item From a matroidal perspective, the class of graphic matroids
is very structured, and thus one might expect there to be nontrivial restrictions on the members of this class. This turns out to be true: for example, we prove that any
$3$-connected graphic matroid has only countable circuits. 
In order to prove this, we found it necessary to work in the topological space representing the matroid.
 
\item Many theorems about infinite graphs, for example certain characterisations of planarity, have matroidal proofs \cite{RD:HB:graphmatroids}. In order to understand when these matroidal techniques are available in infinite graphs, it seems unavoidable to study the whole class of objects to which these proofs apply, namely graphic matroids.


\item Motivated by the fact that many theorems of finite graph theory are true for the topological cycle matroid, people have tried to extend them even further and have started investigating to which topological spaces these theorems might extend. 
As well as particular motivating examples~\cite{DP:dualtrees}, several different candidate classes of topological space have been suggested: graph-like spaces and hereditarily locally connected metric spaces
by Thomassen and Vella \cite{ThomassenVellaContinua}, and
edge spaces by Vella and Richter  \cite{Vella_Richter}. Graph-like continua have been  studied by Christian \cite{christian_phd} and also by Christian, Richter and Rooney \cite{Christian_Richter_Rooney}. This is far from a complete list.

We believe that an important consideration to keep in mind when choosing the appropriate general notion is how well the topological setup interacts with the theory of infinite matroids. This is because many of the graph-theoretic results being generalized can be fruitfully seen, for finite graphs, from a matroidal point of view. In a context where graph-theoretic, matroidal and topological notions can be mutually translated, the techniques from these fields can be powerfully combined. Because our results below suggest it has this feature, we propose the class of objects called ``graph-like spaces inducing  matroids'' below as a good candidate.

 \end{enumerate}


There are at least two reasons for restricting our attention to tame matroids.
From a topological point of view, we want that the circuits should correspond to subspaces which behave like topological circles, and in particular which are compact. So we want circuit-cocircuit intersections to correspond to topological cuts of compact spaces, which are necessarily finite as shown in \cite{BC:determinacy}.

From a matroidal point of view, we want that graphic matroids are representable over every field. 
For this to make sense, we need an appropriate notion of representability for infinite matroids. One strong candidate for such a notion, due to Bruhn and Diestel,  is \emph{thin-sums representability}  \cite{RD:HB:graphmatroids}. Afzali and Bowler have shown~\cite{THINSUMS} that the class of matroids representable in this sense over a given field is well behaved, being closed under duality and taking minors, provided that we restrict our attention to tame matroids. On the other hand, Bowler and Carmesin have shown~\cite{BC:wild_matroids} that if we also consider wild matroids (those which are not tame) then the class of thin-sums representable matroids is not so well behaved. This suggests that we should include tameness in the definition of representability for infinite matroids.

Another consideration supporting this is that, as Bowler, Carmesin and Postle have shown~\cite{BCP:quircuits}, basic representability results such as the equivalence of many standard characterisations of binary matroids fail for some of the simplest examples of wild matroids. Once more, if we restrict our attention to tame matroids then all is well, as Bowler and Carmesin have shown \cite{BC:rep_matroids+}. In the same paper, Bowler and Carmesin made an investigation like that in this paper: they considered the infinite matroids all of whose finite minors are $k$-representable for some finite field $k$. Once again, an essential first step was to restrict attention to tame matroids: once this is done, the class of matroids all of whose finite minors are representable over $k$ is exactly the class of matroids which are thin-sums representable over $k$.

Tameness also plays a fundamental part in Dress's approach to representability of infinite matroids \cite{dress}. Since we want that graphic matroids are thin-sums representable over every field, we require that they are tame.
We don't seem to lose any important scope by restricting our attention to tame matroids: all of the motivating examples are tame, including all matroids that arise from (compactifications of) infinite graphs.

The simplest examples of graphic matroids are the finitary ones: these are precisely the finite cycle matroids of infinite graphs. 
Further examples are given by those matroids whose circuits are (edge sets of) topological circuits of  compactifications of infinite graphs, such as those studied in \cite{RD:HB:graphmatroids}. A similar class of examples is given by matroids whose circuits are topological circles in spaces like the subsets of the plane indicated by the pictures in \autoref{fig:sierpinski_and_ladder}.

\begin{figure}[htbp]
	\centering
	\subfloat[\label{fig:ladder}]
{\includegraphics[width=7cm]{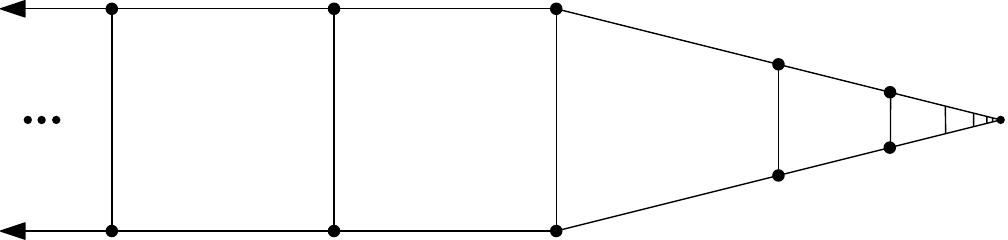}}
	\hspace{1cm}
	\subfloat[\label{fig:sierpinski}]{\includegraphics[width=3cm]{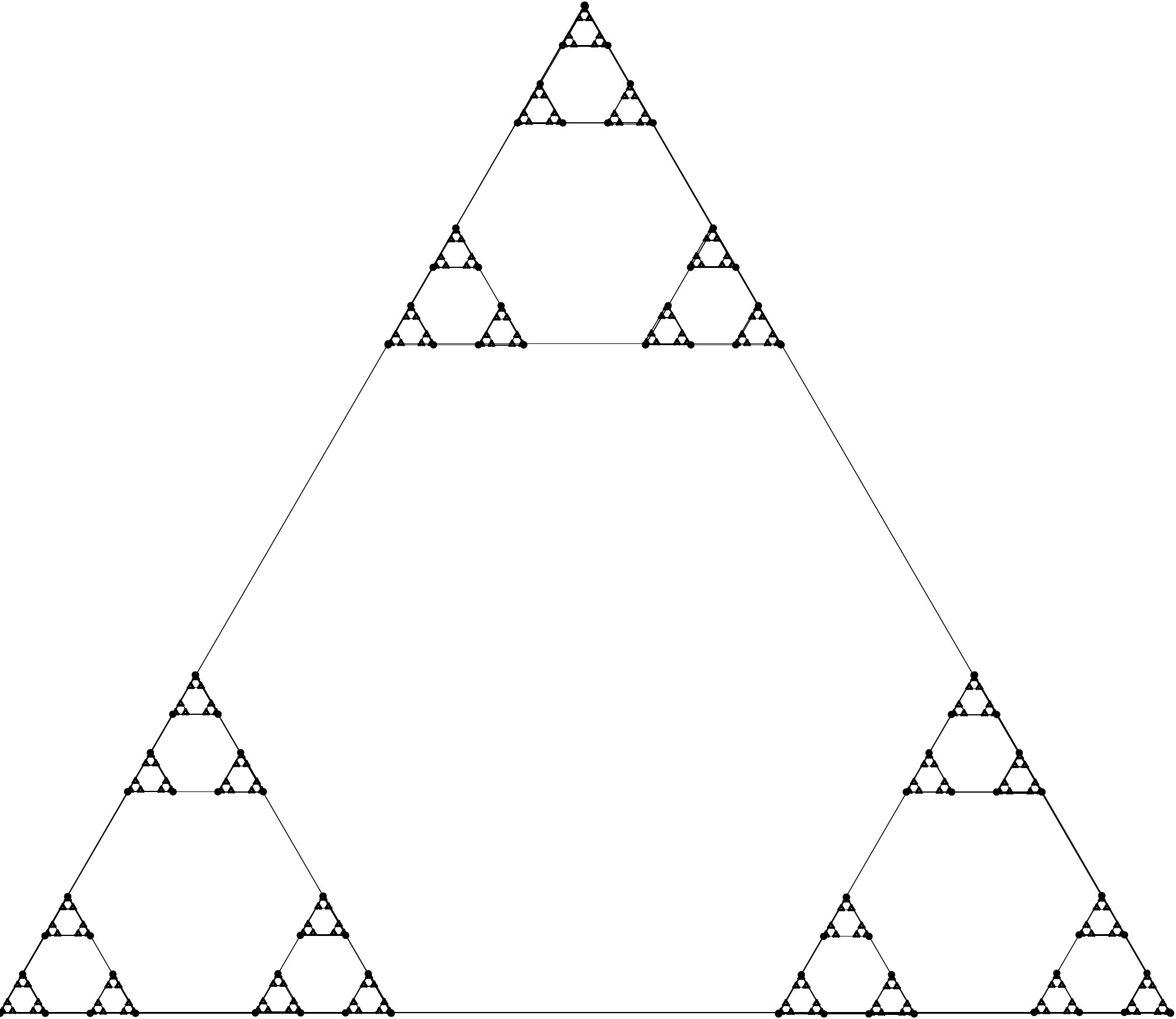}}
	\caption{Subspaces of the plane inducing matroids}
	\label{fig:sierpinski_and_ladder}
\end{figure}
Another rich collection of examples of graphic matroids is constructed in \cite{BC:psi_matroids}.

Our goal is to obtain a characterization similar to those from \cite{BC:rep_matroids}, saying that a matroid has a particular sort of graph-like representation if and only if it is graphic. 
Therefore, we need to find a single unified notion which allows us to represent all graphic
 matroids. We will use graph-like spaces, which were introduced by Thomassen and Vella \cite{ThomassenVellaContinua}. Briefly, for us a graph-like space is a topological space whose points are partitioned into vertices and interior points of edges, and a collection of continuous maps (representing edges) from [0, 1] into that space, sending 0 and 1 to vertices and the other points to interior points of edges, with certain good topological properties. It should be mentioned that although our definition is heavily motivated by that of Thomassen and Vella, it is not quite equivalent.


Before looking at the correspondence between graph-like spaces and matroids, we need to define the circuits and bonds of a graph-like space. When studying cycles of topological spaces the usual approach (see, for example, \cite{RDsBanffSurvey}) has been to consider homeomorphic images of the unit circle. However, this would leave us unable to represent the graphic matroid consisting of a single uncountable circuit. Therefore we use a different, more technical, definition that does allow us to have uncountable circuits. However, it turns out that this is mostly harmless due to the following fact.

\begin{thm}\label{cir_count_thm}
Any circuit in a $3$-connected graphic matroid is countable.
\end{thm}

This countability forces the circuits to be precisely the homeomorphic images of the unit circle.
However, in order to get to \autoref{cir_count_thm}, we rely on topological arguments for which we must use the more technical notion of topological circuits. 

The notion of cut on which we rely is more straightforward. Consider a pair $(U,V)$ of disjoint open subsets of a graph-like space, partitioning the vertices. The set of edges with one endpoint in each of $U$ and $V$ is a {\em topological cut} of the graph-like space, and the minimal non-empty topological cuts of a graph-like space are its {\em topological bonds}. We say that a graph-like space $G$ induces a matroid $M$ if the topological circuits and bonds of $G$ are exactly the same as the circuits and cocircuits of $M$. 

We define minor operations (contraction and deletion) on graph-like spaces that correspond to the minor operations on matroids (if $G$ induces $M$ then $G/C\backslash D$ induces $M/C\backslash D$). Since graph-like spaces with finitely many edges are (essentially) finite graphs, it follows that if $G$ induces a matroid $M$, then any finite minor of $M$ is graphic. In addition, as noted above, compactness of the topological circuits forces any matroid induced by a graph-like space to be tame.

So any matroid induced by a graph-like space is graphic, and our other main result is that the converse is also true: 

\begin{thm}
A matroid  is induced by some graph-like space if and only if it is graphic.
\end{thm}

 To prove the converse implication, we use a compactness argument similar to that in \cite{BC:rep_matroids}. For this we need a structure of finite character that can be imposed on all finite graphic matroids, such that if a matroid $M$ has such a structure then we can build a graph-like space inducing $M$. We introduce such a structure, which we call a {\em graph framework}.

Graph-like spaces in general may be topologically quite nasty.
However, we are able to prove the following.
\begin{thm}
Let $M$ be a $3$-connected graphic matroid. Then there exists a graph-like space $G$ inducing $M$ with the following properties.
\begin{enumerate}
 \item Let $s$ be an $M$-base. The closure of $s$ is connected and path-connected.
For any $e\in s$, the subspace $s-e$ is disconnected.
\item $G$ is locally path-connected.
\item $G$ is regular. 
\end{enumerate}
\end{thm}

The remainder of this paper is organized as follows. In \autoref{sec2}, we give background information on infinite matroids, and prove a few simple lemmas. \autoref{sec3} introduces graph-like spaces, and \autoref{sec4} covers pseudo-arcs and pseudo-circles, analogues in a graph-like space of paths and cycles in finite graphs. 

\autoref{sec5} is concerned with which graph-like spaces induce matroids. In \autoref{sec6}, we define graph frameworks, and show by compactness that a  graphic matroid has a graph framework. We then show how to construct a graph-like space that induces a graphic matroid with a given graph framework. 
In \autoref{sec8}, we study the topological properties of the graph-like space we have constructed.

Finally, in \autoref{sec9}, we prove that $3$-connected matroids induced by a graph-like space have no uncountable circuits. In \autoref{sec10}, we finish with an open problem.

This is the first in a series of 2 papers: the second will focus on cofinitary graphic matroids, which correspond to compact graph-like spaces. In this context, we are able to derive some stronger results. Any compact graph-like space induces a cofinitary graphic matroid. For any 3-connected cofinitary graphic matroid, there is a unique connected compact graph-like space representing it (this is the infinite analogue of a theorem of Whitney). Furthermore, the countable cofinitary graphic matroids are precisely those arising as minors of the matroids induced by Freudenthal compactifications of countable graphs.

\section{Preliminaries}\label{sec2}

Throughout, notation and terminology for (infinite) graphs are those of~\cite{DiestelBook10}, and for matroids those of~\cite{Oxley,matroid_axioms}.

$M$ always denotes a matroid and $E(M)$ (or just $E$), $\Ical(M)$ and $\Ccal(M)$ denote its ground 
set and its sets of independent sets and circuits, respectively. For the remainder of this section we shall recall some basic facts about infinite matroids.

A set system $\Ical\subseteq \Pcal(E)$ is the set of independent sets of a matroid if and only if it satisfies the following {\em independence  axioms\/} \cite{matroid_axioms}.
\begin{itemize}
	\item[(I1)] $\varnothing\in \Ical(M)$.
	\item[(I2)] $\Ical(M)$ is closed under taking subsets.
	\item[(I3)] Whenever $I,I'\in \Ical(M)$ with $I'$ maximal and $I$ not maximal, there exists an $x\in I'\setminus I$ such that $I+x\in \Ical(M)$.
	\item[(IM)] Whenever $I\subseteq X\subseteq E$ and $I\in\Ical(M)$, the set $\{I'\in\Ical(M)\mid I\subseteq I'\subseteq X\}$ has a maximal element.
\end{itemize}

A set system $\Ccal\subseteq \Pcal(E)$ is the set of circuits of a matroid if and only if it satisfies the following {\em circuit  axioms\/} \cite{matroid_axioms}.
\begin{itemize}
\item[(C1)] $\varnothing\notin\Ccal$.
\item[(C2)] No element of $\Ccal$ is a subset of another.
\item[ (C3)](Circuit elimination) Whenever $X\subseteq o\in \Ccal(M)$ and $\{o_x\mid x\in X\} \subseteq \Ccal(M)$ satisfies $x\in o_y\Leftrightarrow x=y$ for all $x,y\in X$, 
then for every $z \in o\setminus \left( \bigcup_{x \in X} o_x\right)$ there exists a  $o'\in \Ccal(M)$ such that $z\in o'\subseteq \left(o\cup  \bigcup_{x \in X} o_x\right) \setminus X$.

\item[(CM)] $\Ical$ satisfies (IM), where $\Ical$ is the set of those subsets of $E$ not including an element of $\Ccal$.
\end{itemize}

For a base $s$ of a matroid $M$, and $e \in E \sm s$, there is a unique circuit $o_e$ with $e \in o_e \se s + e$. We call this circuit the {\em fundamental circuit} of $e$ with respect to $s$. Similarly, for $f \in b$ we call the unique cocircuit $b_f$ with $f \in b_f \se (E \sm s) + f$ the {\em fundamental cocircuit} of $f$ with respect to $s$.

The following straightforward Lemmas can be proved as for finite matroids (see, for example, \cite{BC:rep_matroids+}).

\begin{lem}\label{fdt}
 Let $M$ be a matroid and $s$ be a base.
Let $o_e$ and $b_f$ a fundamental circuit and a fundamental cocircuit with respect to $s$, then
\begin{enumerate}
 \item $o_e\cap b_f$ is empty or $o_e\cap b_f=\{e,f\}$ and
\item $f\in o_e$ if and only if $e\in b_f$.
\end{enumerate}
\end{lem}

\begin{lem}\label{o_cap_b}
 For any circuit $o$ containing two edges $e$ and $f$, there is a cocircuit $b$ such that $o\cap b=\{e,f\}$.
\end{lem}

\begin{lem}\label{almost_fundamental_cocircuit}
 Let $I$ be some independent set in some matroid $M$.
Then for each $e\in I$ there is a cocircuit $b$ meeting $I$
precisely in $e$
\end{lem}

\begin{lem} \label{rest_cir}
 Let $M$ be a matroid with ground set $E = C \dot \cup X \dot \cup D$ and let $o'$ be a circuit of $M' = M / C \backslash D$.
Then there is an $M$-circuit $o$ with $o' \subseteq o \subseteq o' \cup C$.
\end{lem}

\begin{lem}\label{is_scrawl}
Let $M$ be a matroid, and let $w\subseteq E$. The following are equivalent:
\begin{enumerate}
 \item $w$ is a union of circuits of $M$.
 \item $w$ never meets a cocircuit of $M$ just once.
\end{enumerate}
\end{lem}

The basic theory of infinite binary matroids is introduced in~\cite{BC:rep_matroids+}. One characterisation of such matroids given there is that every intersection of a circuit with a cocircuit is both finite and of even size.
\begin{lem}\label{binary_x}
Let $M$ be a binary matroid and $X\se E(M)$ with the property that it meets every circuit finitely and evenly. Then $X$ is a disjoint union of cocircuits.
\end{lem}

\begin{proof}
By Zorn's Lemma, we can pick $Y\se X$ maximal with the property that it is a disjoint union of cocircuits. As $Y\se X$, the set $Y$ meets every circuit finitely, and so meets every circuit evenly. By the choice of $Y$, the set $X\sm Y$ does not include a circuit.
But $X\sm Y$ meets every circuit evenly, and so is empty by the dual of \autoref{is_scrawl}.
 This completes the proof.
\end{proof}

\begin{lem}\label{magic_lemma}
Suppose that $M$ is a matroid, and ${\cal C}$, ${\cal C}^*$ are collections of subsets of $E(M)$ such that ${\cal C}$ contains every circuit of $M$, ${\cal C}^*$ contains every cocircuit of $M$, and for every $o\in {\cal C}$, $b\in {\cal C}^*$, $|o\cap b|\neq 1$. Then the set of minimal nonempty elements of ${\cal C}$ is the set of circuits of $M$ and the set of minimal nonempty elements of ${\cal C}^*$ is the set of cocircuits of $M$.
\end{lem}

\begin{proof}
The conditions imply that no element of ${\cal C}$ ever meets a cocircuit of $M$ just once, so every element of $\Ccal$ is a union of circuits of $M$ by \autoref{is_scrawl}. Since every circuit of $M$ is in $\Ccal$, the minimal nonempty elements of $\Ccal$ are precisely the circuits of $M$. The other claim is obtained by a dual argument.
\end{proof}

A {\em switching sequence} for a base $s$ in a matroid with ground set $E$ is a finite sequence $(e_i | 1 \leq i \leq n)$ whose terms are alternately in $s$ and not in $s$ and where for $i<n$ if $e_i \in s$ then $e_{i + 1} \in b_{e_i}$ and if $e_i \not \in s$ then $e_{i+1} \in o_{e_i}$.

\begin{lem}\label{swseq}
Let $M$ be a connected matroid with a base $s$, and $e$ and $f$ be edges of $M$. Then there is a switching sequence with first term $e$ and last term $f$.
\end{lem}
\begin{proof}
Let $e$ be any edge of $M$, and let $X$ be the set of those $f \in E(M)$ for which there is such a switching sequence. Then $s \cap X$ is a base for $X$, since for any $f \in X \sm s$ we have $o_f \subseteq X$. Similarly, $s \sm X$ is a base for $E(M) \sm X$, since for any $f \in E(M) \sm X\sm s$ and any $g \in o_f$ we have $f \in b_g$ by \autoref{fdt} and so $g \not \in X$. Thus $X$ and $E(M) \sm X$ form a separation of $M$, and since $M$ is connected this means that $X$ must be the whole of $E$, completing the proof.
\end{proof}

A {$k$-separation} of a matroid $M$ is a partition $(A, B)$ of the ground set of $M$ such that each of $A$ and $B$ has size at least $k$ and there are bases $s_A$ and $s_B$ of $A$ and $B$ and $s$ of $A$ such that $|s_A \cup s_B \setminus s| < k$. A 1-separation may also be called a {\em separation}. A matroid without $l$-separations for any $l < k$ is {\em $k$-connected}. A matroid is {\em connected} if it is 2-connected. Connected matroids can equivalently be characterised as those in which any 2 distinct edges lie on a common circuit \cite{BW:mat_con}.

\section{Graph-like spaces}\label{sec3}

The key notion of this section is the following, which is based on a definition from \cite{ThomassenVellaContinua}:

\begin{dfn}\label{def:gls}
A \emph{ graph-like space $G$} is a 
topological space (also denoted $G$) together with 
a \emph{vertex set} $V=V(G)$, an \emph{edge set} $E=E(G)$ and for each $e \in E$ a continuous map $\iota^G_e \colon [0, 1] \to G$ (the superscript may be omitted if $G$ is clear from the context) such that:
\begin{itemize}
\itum The underlying set of $G$ is $V\sqcup [(0,1)\times E]$
\itum For any $x \in (0, 1)$ and $e \in E$ we have $\iota_e(x) = (x, e)$.
\itum $\iota_e(0)$ and $\iota_e(1)$ are vertices (called the {\em endvertices} of $e$).
\itum $\iota_e \restric_{(0, 1)}$ is an open map.
\itum For any two distinct $v, v' \in V$, there are disjoint open subsets $U, U'$ of $G$ partitioning $V(G)$ and with $v \in U$ and $v' \in U'$.
\end{itemize}

\emph{The inner points of the edge $e$} are the elements of $(0,1)\times \{e\}$.
\end{dfn}

Note that $V(G)$, considered as a subspace of $G$, is totally disconnected, and that $G$ is Hausdorff.

Let $e$ be an edge in a graph-like space with $\iota_e(0) \neq \iota_e(1)$. Then $\iota_e$ is a continuous injective map from a compact to a Hausdorff space and so it is a homeomorphism onto its image. The image is compact and so is closed, and therefore is the closure of $(0, 1) \times \{e\}$ in $G$. So in this case $\iota_e$ is determined by the topology of $G$. The same is true if $\iota_e(0) = \iota_e(1)$: in this case we can lift $\iota_e$ to a continuous map from $S^1 = [0, 1]/(0 = 1)$ to $G$, and argue as above that this map is a homeomorphism onto the closure of $(0, 1) \times \{e\}$ in $G$.
In this case, we say that $e$ is a loop of $G$.

Next we shall define maps of graph-like spaces. Let $G$ and $G'$ be graph-like spaces.
Two maps $\varphi_V:V(G)\to V(G')$ and $\varphi_E:E(G)\to (E(G')\times \{+,-\})\sqcup V(G)$
induce a function $\varphi$ sending points of $G$ to points of $G'$ as follows: a vertex $v$ of $G$ is mapped to $\varphi_V(v)$. Let $e$ be an edge, and $(r,e)$ one of its interior points.
If $\varphi_E(e)$ is a vertex, then $(r,e)$ is mapped to $\varphi_E(e)$.
If $\varphi_E(e)=(f,+)$ for some $f\in E(G')$, then $(r,e)$ is mapped to $(r,f)$.
Similarly, if $\varphi_E(e)=(f,-)$ for some $f\in E(G')$, then $(r,e)$ is mapped to $(1-r,f)$.
If a function arising in this way is continuous we call it a \emph{map of graph-like spaces}.
From this definition, it follows that if $v$ is an endvertex of $e$, then $\varphi(v)$ is either an endvertex of or equal to the image of $e$.

Let us consider some examples of graph-like spaces. 
We shall write $[0,1]$ for the unique graph-like space without loops having precisely one edge and two vertices.
There are exactly seven maps of graph-like spaces from $[0,1]$ to two copies of $[0,1]$ glued together at a vertex: four of these have one of the copies of $[0,1]$ as their image and the other three map the whole interval to a vertex.
However, none of these maps is bijective nor has an inverse, even though the underlying topological spaces are homeomorphic.

Figures \ref{fig:ladder} and \ref{fig:sierpinski} from the introduction define graph-like spaces with vertices and edges as in the figures. In each case the topology is that induced by the embedding in the plane suggested by the figures.
For a locally finite graph $G=(V,E)$, the topological space $|G|$ is a graph-like space with
vertex set $V\cup \Omega(G)$ and edge set $E$ (see \cite{DiestelBook10} for the definition of $|G|$).
Note that if $G$ is finite, then $|G|$ is homeomorphic to the geometric realisation of $G$ considered as a simplicial complex.

\begin{lem} \label{finite_gls_is_graph}
Let $G$ be a graph-like space with only finitely many edges and finitely many vertices.
Then $G$ is 
homeomorphic to $|H|$ for some finite graph $H$.
\end{lem}

\begin{proof}
$G$ is compact, since it is a union of finitely many compact subspaces.
Let $H$ be the graph with edge set $E(G)$ and vertex set  $V(G)$, and in which $v$ is an endpoint of $e$ if and only if this is true in $G$.
We now construct a map $\varphi \colon G \to |H|$ as follows:
taking $\varphi_V$ to be the identity and $\varphi_E$ to be the function sending each edge $e$ to $(e,+)$, we build $\varphi$ as in the definition of a map of graph-like spaces. 

It remains to show that the function $\varphi$ is continuous:
since it is a bijection from 
a compact to a Hausdorff space, it will then be a homeomorphism. 
We begin by noting that for any $e \in E(G)$, the restriction of $\varphi$ to the image of $\iota^G_e$ is a homeomorphism, by the remarks following \autoref{def:gls}.
Now we need to show for any $x\in |H|$ that the inverse image of any open neighbourhood $U$ of $\varphi(x)$ includes an open neighbourhood of $x$.
If $x$ is an interior point of an edge, this is clear.
Otherwise, $x$ is a vertex of $|H|$. Then there is an open neighbourhood $U' \se U$ of $x$ which only meets edges incident with $x$. For each such edge $e$, since the restriction of $\varphi$ to the image of $\iota^G_e$ is a homeomorphism, there is an open set $V_e$ of $G$ with $V_e \cap \im(\iota_e^G) = \varphi^{-1}(U') \cap \im(\iota_e^G)$. Letting $V$ be the intersection of the $V_e$, we obtain that $V$ is an open neighbourhood of $x$ included in $\varphi^{-1}(U)$, completing the proof that $\varphi$ is continuous.
\end{proof}

All the above examples of graph-like spaces will turn out to induce matroids.
Before we can make this more explicit, we must first introduce the notions of topological circuits and bonds in a graph-like space. The discussion of topological circuits will be delayed until the next section, but we will introduce topological bonds now.

\begin{dfn}\label{def:top_bond}
Given a pair of disjoint open subsets of a graph-like space $G$ partitioning the vertices, 
we call the set of those edges having an endvertex in both sets
\emph{a topological cut} of $G$.
A \emph{topological bond} of $G$ is a minimal nonempty topological cut of $G$.
\end{dfn}

Given a graph-like space $G$ and a set of edges $R \se E(G)$, we define the graph-like space $G\restric_R$, the {\em restriction} of $G$ to $R$, to have the same vertex set as $G$ and edge set $R$. Then the ground set of $G \restric_R$ is a subset of that of $G$, and we give it the subspace topology.  Evidently, for any topological cut $b$ of $G$, $b \cap R$ is a topological cut of $G \restric_R$.
The \emph{deletion} of $D$ from $G$, denoted by $G\backslash D$, is $G\restric_{(E\sm D)}$.  
We abbreviate $G\backslash \{e\}$ by $G-e$. The inclusion map $g_D$ from $G\backslash D$ to $G$ is a map of graph-like spaces.

Note that $G \restric R$ has the same vertex set as $G$, even though only the vertices in the closure of $(0,1) \times R$ play an important role in the new space. By analogy to the notation of \cite{DiestelBook10}, we also introduce a notation for the graph-like space whose edges are those in $R$ but whose vertices are those in the closure of $(0, 1) \times R$. We will call this subspace the \emph{standard subspace with edge set $R$}, and denote it $\overline R$. 

Given a graph-like space $G$ and $C\subseteq E(G)$, 
we define the \emph{contraction} $G/C$ of $G$ onto $C$ as follows:

Let $\equiv_C$ be the relation on the vertices of $G$ defined by $u\equiv_C v$ if every topological cut with $u$ and $v$ in different parts meets $C$. It is easy to check that $\equiv_C$ is an equivalence relation. The vertex set of $G/C$ is the set of $\equiv_C$-equivalence classes, and the edge set is $E(G)\setminus C$.

It remains to define the topology of $G/C$.
We shall obtain this as the quotient topology derived from a function $f_C\colon G\to G/C$, to be defined next.

The function $f_C$ sends each vertex to its $\equiv_C$-equivalence class and is bijective on the interior points of edges of $E\sm C$. The two endpoints of an edge in $C$ are in the same equivalence class, and we send all of its interior points to that equivalence class.

Taking this quotient topology ensures that $G/C$ is a graph-like space, and makes $f_C$ a map of graph-like spaces.
In $G/C$, the endpoints of an edge are the equivalence classes of its endpoints in $G$. For any topological cut $b$ of $G$ with $b \cap C = \emptyset$, the two sides of $b$ are closed under $\equiv_C$ by definition, and so $b$ is also a topological cut in $G/C$.
 
We define $G.X:=G/(E\sm X)$ and $G/e:=G/\{e\}$.
It is straightforward to check for disjoint sets $C$ and $D$ that $(G\backslash D)/C$ and $(G/C)\backslash D$ are equal and the following diagram commutes.

$$\xymatrix{G \backslash D \ar[d]_{f_C} \ar[r]^{g_D} & G \ar[d]^{f_C} \\ G / C \backslash D \ar[r]_{g_D} & G / C \\}$$

Contraction behaves especially well when applied to one side of a topological cut.

\begin{lem}\label{cont_cut}
Let $U$ and $V$ be disjoint open sets inducing a topological cut of a graph-like space $G$. Let $C$ be the set of edges of $G$ that have both end-vertices in $U$. Then no pair of distinct vertices $u$ and $v$ of $G$ with $v \in V$ get identified in the contraction $G/C$
\end{lem}
\begin{proof}
Since $u$ and $v$ are distinct vertices of $G$, there are disjoint open sets $U'$ and $V'$ partitioning the vertices of $G$ with $u \in U'$ and $v \in V'$. Then $(U \cup U', V \cap V')$ is a pair of disjoint open subsets of $G$ partitioning the vertices of $G$. The cut it induces does not meet $C$, and it has $u$ and $v$ on opposite sides, so it witnesses that $u$ and $v$ are not identified in $G/C$.
\end{proof}

Thus the singleton $\{v\}$ of each vertex $v$ in $V$ gives a vertex of $G/C$. To simplify our notation, we will at times write as though the vertices in $G/C$ on the $V$ side of the cut were the same as those of $G$, rather than their singletons. 

It is clear that after contracting all the edges on one side of a bond there remain no edges with both end-vertices on that side of the bond. In this case, we can show that there is only one interesting vertex left on that side of the bond.

\begin{lem}\label{id_bond}
Let $U$ and $V$ be disjoint open sets inducing a topological bond $b$ of a graph-like space $G$. Suppose that there are no edges of $G$ with both end-vertices in $V$. Then there is at most one vertex of $V$ which is an endpoint of an edge of $G$.
\end{lem}
\begin{proof}
Suppose that there are two such vertices $v_e$, an endpoint of $e$, and $v_f$, an endpoint of $f$. Since $v_e \neq v_f$, there are disjoint open subsets $U'$ and $V'$ of $G$ partitioning the vertices with $v_e \in U'$ and $v_f \in V'$. Then $U \cup U'$ and $V \cap V'$ induce a topological cut $b'$ of $G$. Furthermore, $b'$ is nonempty since $f \in b'$, is a subset of $b$ since any edge of $b'$ has one endpoint in $V$, and is a proper subset since $e \in b \setminus b'$. This contradicts the minimality of $b$. 
\end{proof}

\section{Pseudoarcs and Pseudocircles}\label{sec4}

When investigating a topological space, it is common to consider arcs in that space, that is, continuous injections from the unit interval to that space. We must consider maps from a slightly more general kind of domain.
A \emph{pseudo-line} is a compact and connected graph-like space together with a start-vertex 
$s$ and an end-vertex $t$ satisfying the following.
\begin{enumerate}
\item Removing any edge separates $s$ from $t$.
\item Any two vertices can be separated by removing a single edge.
\end{enumerate}

Note that 1 implies that $s\neq t$ provided that the pseudo-line has an edge.
A \emph{pseudo-path from $v$ to $w$ in a graph-like space $G$} is 
a map of graph-like spaces from a pseudo-line to $G$ sending the start-vertex to $v$
and the end-vertex to $w$.
A \emph{pseudo-arc} is an injective pseudo-path.
Note that any pseudo-arc is a homeomorphism onto its image since the 
domain is compact and the codomain is Hausdorff. Thus we will also refer to the images of pseudo-arcs as pseudo-arcs. In particular, a {\em pseudo-arc in} a graph-like space $G$ is the image of such a map (in other words, it is a subspace of $G$ which is also a pseudo-line).

For example, the unit interval $[0,1]$ is a pseudo-line 
with one edge and two vertices.
Note however that not every path in a graph-like space
is necessarily a pseudo-path since a pseudo-path must be a map of graph-like spaces.
For example, $f:[0,1]\to[0,1]$ with $f(x)=x(1-x)$ is a path but not a pseudo-path.

Any arc joining two vertices in a graph-like space may be given the structure of a pseudo-arc,
since its image is a pseudo-line. All pseudo-lines arising in this way are homeomorphic to the unit interval but not 
every pseudo-line has this property.
In fact, we shall show that for any linearly ordered set $P$ we can build an associated pseudo-line $L(P)$
whose edge set is $P$. In particular, there are pseudo-lines of arbitrarily large cardinality.

To construct $L(P)$, we take as our vertex set $V$ the set of initial segments of $P$. 
Next, we take a subbasis of the topology to consist of 
the sets of the type $S(p,r)^+$ or $S(p,r)^-$ defined below.

For every  $p\in P$ and $r\in (0,1)$, let $S(p,r)^-$
contain precisely those vertices which do not contain $p$.
Furthermore, let  $S(p,r)^-$ contain all interior points of edges $x$
with $x < p$ together with $(0,r) \times \{p\}$.

Similarly, let $S(p,r)^+$
contain precisely those vertices which contain $p$.
Furthermore, let  $S(p,r)^+$ contain all interior points of edges $x$
with $x > p$ together with $(r,1) \times \{p\}$.
For the start-vertex we take $\emptyset$ and for the end-vertex we take $P$.

\begin{claim}\label{lin_are_pline}
The space $L(P)$ defined above is a pseudo-line for any linearly ordered set $P$.
\end{claim}

\begin{proof}
It is clear that $L(P)$ is a graph-like space
where the edge $p$ has end vertices $\iota_{p}(0)=\{x\in P|x<p\}$
and $\iota_{p}(1)=\{x\in P|x\leq p\}$.
Furthermore, it is clear that $L(P)$ satisfies conditions 1 and 2 in the definition of pseudo-lines.

To show that $L(P)$ is a pseudo-line, it remains to check that it is connected and compact.
For the  connectedness, 
let $U$ be an open and closed set containing the start-vertex $\emptyset$.
Since for any edge $e$ the subspace topology of $\iota_e([0,1])$ is that of $[0,1]$, which is connected,
the set $\iota_e([0,1])$ is either completely included in $U$ or disjoint from $U$.
Let $v=\{p\in P|S(p,1/2)^-\subseteq U\}$.
Then the vertex $v$ is in $U$ since any neighbourhood of it meets $U$ (even if $v=\emptyset$).  
So since $U$ is open, it includes an open neighbourhood $O$ of $v$.
Since by our earlier remarks $U$ includes all edges $p \in v$ and so also all vertices $w \se v$, we may assume without loss of generality that either $v = P$ or else  $O$ has the form $S(p, r)^-$ for some $p \not \in v$. In the second case we conclude that $p\in v$, which is impossible.
Hence $v=P$. Since the closure of $\bigcup_{p\in P}\iota_p((0,1))$
is the whole of $L(P)$, the closed set $U$ is the whole of $L(P)$.
Hence $L(P)$ is connected, as desired.

It remains to show that $L(P)$ is compact.
By Alexander's theorem, it suffices to check that any open cover by subbasic open
elements has a finite subcover.
Let $L(P)=\bigcup_{i\in I^+} S(p_i,r_i)^+\cup \bigcup_{i\in I^-} S(p_i,r_i)^-$ be an open cover by subbasic open sets.
Let $v=\{p\in P|\exists i\in I^-: p< p_i  \}$.

First we consider the case where there is some $i\in I^+$ with $v\in S(p_i,r_i)^+$.
Then $p_i\in v$, so there is some $j\in I^-$ such that $p_i<p_j$.
This means that $S(p_i,r_i)^+$ and $S(p_j,r_j)^-$ cover $L(P)$.

Otherwise there is some $i\in I^-$ with $v\in S(p_i,r_i)^-$.
Then $p_i\notin v$ and so $p_i$ is maximal amongst the $p_j$ with $j\in I^-$.
Thus $v+ p_i$ is contained in some $S(p_k,r_k)^+$ with $k\in I^+$.
Then $S(p_i,r_i)^-$ and $S(p_k,r_k)^+$, together with some finite 
collection of sets from our cover covering the compact subspace $\iota_{p_i}([0,1])$,
form a finite subcover, completing the proof.
\end{proof}

\begin{eg}
 If $P=\omega_1$, then $L(P)$ is the {\em long line}, which is not homeomorphic to $[0,1]$.
\end{eg}

We shall show below that if a nontrivial pseudo-line
has countably many edges, then it is homeomorphic to the unit interval.
This will follow from the fact that every pseudo-line can be constructed from a linear order in the above way.
To get this, we define for every pseudo-line $L$ a linear order on its edges via
$e<_L f$ if and only if $e$ is in the component of $L-f$ containing the start-vertex.

\begin{lem}\label{anti}
 $e<_Lf$ if and only if $f$ is in the  component of $L-e$ containing the end-vertex of $L$.
In particular, $<_L$ is antisymmetric.
\end{lem}

\begin{proof}
Let $s$ and $t$ be the start- and end-vertices of $L$, respectively.
For any nontrivial topological separation $(U, U')$ of $L -  f$, both sides of the topological separation must meet the closure of $\iota_f((0,1))$. If for example $U$ did not, then $(U, U' \cup \iota_f((0,1)))$ would be a nontrivial topological separation of $L$, since $\iota_f((0,1))$ is open. So one side must contain $\iota_f(0)$ and the other must contain $\iota_f(1)$.
Thus there are at most two components of $L-f$, and there are
exactly two since $L-f$ is disconnected.

Each of $\iota_f(0)$ and $\iota_f(1)$ lies in one of these components.
Let $k_s$ be the one in the same component as $s$ and $k_t$ the one in the same component as $t$.

If $e<_Lf$, then $e$ is in the other component than $k_t$.
So even in $L-e-f$, the vertices $t$ and $k_t$ are in the same component.
Hence in $L-e$, the edge $f$ is in the same component as $t$, since it is in the same component as $k_t$.
The reverse implication is symmetric, completing the proof.
\end{proof}

\begin{lem}\label{trans}
The relation  $<_L$ is transitive.
\end{lem}

\begin{proof}
Let $e<_Lf$ and $f<_Lg$ and let $s$ and $t$ be the start- and end-vertices of $L$, as above.
By \autoref{anti}, in $L-f$ the edge $e$ is in the same component as $s$ whereas $g$ is in the same component as $t$. This means that in $L-g$ the edge $e$ is still in the same component as $s$, as desired.
\end{proof}

\begin{lem}\label{L_is_LP}
If $P$ is linearly ordered by $<$, then $<$ and the linear order $<_{L(P)}$ defined above
coincide.
\end{lem}

\begin{proof}
This follows from the fact that for any $p\in P$ the two connected components of $L(P)-p$
are $S(p,1/2)^-\sm ((0,1/2)\times \{p\})$ and  $S(p,1/2)^+\sm ((1/2,1)\times \{p\})$.
\end{proof}

\begin{lem}\label{pline_lin_order}
For any pseudo-line $L$, we have $L\cong L((E(L),<_L))$.
\end{lem}

\begin{proof}
Without loss of generality, each edge $e$ of $L$ is parametrised in such a way that
in $L-e$ its start-vertex $\iota_e(0)$ is in the same component as $s$.
We define a map $\varphi:L\to L((E(L),<_L))$ of graph-like spaces
which is the identity on edges.
A vertex $v$ is mapped to the set of edges $e$ of $L$ such that $v$ is in the same component of $L - e$ as $t$ is.
This is an initial segment by a similar argument to that in the proof of \autoref{trans}.

Since $L$ is compact and  $(E(L),<_L))$ is Hausdorff, it suffices to show that $\varphi$ is continuous and bijective.

For continuity, it suffices to show that the inverse image of every subbasic open set is open. First, consider a subbasic open set $S(e,r)^-$.
The inverse image of $S(e,r)^-$ is by definition the component $C$ of $L-e$ containing $s$ together
with $\iota_e((0,r))$.
By the definition of the subspace topology, there is an $L$-open  set $O$ including $C$ 
that is included in $C\cup \iota_e((0,1))$.
Since $L$ is a graph-like space, the set $\iota_e([r/2,1])$ is closed.
Thus $(O\sm \iota_e([r/2,1]))\cup \iota_e((0,r))$ is open and equal to $\varphi^{-1}(S(e,r)^-)$.
Similarly, one shows that each set of the form $\varphi^{-1}(S(e,r)^+)$ is open.
This completes the proof that $\varphi$ is continuous.

To see that $\varphi$ is injective, note that any two vertices of $L$ can be separated by some edge, thus cannot have the same image.

For surjectivity, let $v$ be a vertex of $(E(L),<_L))$.
If $v$ is the start-vertex then it is $\varphi(s)$. If it is the end-vertex then it is $\varphi(t)$. Otherwise, let $U^-$ be the union of the sets $S(e,r)^-$ with $e\in v$ and $r\in (0,1)$ and let $U^+$ be the union of the sets $S(e,r)^+$ with $e \not \in v$ and $r\in (0,1)$. Then $(U^-, U^+)$ is a nontrivial topological separation of $L((E(L), <_L)) - v$, with interior points of edges on both sides, so $(\varphi^{-1}(U^-), \varphi^{-1}(U^+))$ is a nontrivial topological separation of $L \setminus \varphi^{-1}(\{v\})$, so $\varphi^{-1}(\{v\})$ is nonempty.

\end{proof}

\begin{cor}
Any nontrivial pseudo-line is the closure of the set of interior points of its edges. Any nontrivial pseudo-arc in a graph-like space is the standard subspace corresponding to its set of edges.
\end{cor}

\begin{rem}
Contracting a set of edges of a pseudo-line $L(P)$ corresponds to removing that set of edges from the associated poset $P$.
\end{rem}

\begin{cor}\label{line_contract}
Any contraction of a pseudo-line is a pseudo-line.
\qed
\end{cor}

\begin{cor}\label{count_line}
Any nontrivial pseudo-line with only countably many edges is homeomorphic to the unit interval.
\end{cor}

\begin{proof}
Consider the linear order $\Qbb\times (E(L),<_L)$.
This is dense, countable and has neither a largest nor a smallest element.
Since the theory of such linear orders is countably categorical, this
order is isomorphic to the order of the rationals in $(0,1)$. Pick such an isomorphism
and call it $\varphi$.

Next, we define a pseudo-line $L'$ homeomorphic to $[0,1]$ whose edges have the same linear order as $L$. By \autoref{pline_lin_order} this will then show that $L$ is homeomorphic to  $[0,1]$.

The edge set of $L'$ is given by $E(L)$. The vertices are those points in $[0,1]$
that do not lie in between $\varphi(q,e)$ and $\varphi(q',e)$ for any $e\in E(L)$ and $q,q'\in \Qbb$.
Now we give a bijection $\psi$ between the underlying set of $L'$ and $[0,1]$.
We will use this bijection to induce the topology on $L'$.
On the vertices we choose this function to be the identity. 
For any $e\in E(L)$ we map $(0,1)\times\{e\}$ to $[0,1]$ by sending $(r,e)$ to 
$r\cdot \sup_{q\in \Qbb} \varphi(q,e)+ (1-r) \inf_{q\in \Qbb} \varphi(q,e)$.

It is straightforward to check that $L'$ is a pseudo-line and $<_{L'}$ coincides with $<_L$. 
This completes the proof.
\end{proof}

\begin{lem}\label{components_expl_arc}
Let $s_1<_L\ldots <_L s_n$ be finitely many edges of a pseudo-line $L$.
Let $S=\bigcup_{i=1}^n \iota_{s_i}((0,1))$.
Then $L\sm S$ has $n+1$ components each of which is a pseudo-line.
These are $S(s_1,1/2)^-\sm S$, and $S(s_{i+1},1/2)^-\cap S(s_{i},1/2)^+)\sm S$
for $1\leq i\leq n-1$ and $S(s_n,1/2)^+\sm S$.
\end{lem}

\begin{proof}
The assertion follows by induction from the following.
Let $e\in L$. Then $L-e$ has two components that are both pseudo-arcs.
These are $S(e,1/2)^-\sm ((0,1)\times \{e\})$ and $S(e,1/2)^+\sm ((0,1)\times \{e\})$. 
This follows from the combination of \autoref{L_is_LP} and \autoref{pline_lin_order}.
\end{proof}

We shall write $\leq_L$ for the linear order obtained from $<_L$ by adding all the relations
$(e,e)$ for every $e\in E(L)$.
Having proved \autoref{pline_lin_order}, we get a linear order on all points of $L$
via $x\leqq_L y$ if in $L-y$ the point $x$ lies in the same component as the start-vertex of $L$.

\begin{lem}\label{min_open}
Let $x$ be a point of a pseudo-line $L$.
The set of all points in $L$ that are $\leqq_L$-smaller than $x$ 
and not equal to $x$ form an open set.
\end{lem}

\begin{proof}
 By the definition of $\leqq_L$ the set of such points is a connected component of $L-x$, so by \autoref{pline_lin_order} it is open in $L-x$, and since $\{x\}$ is closed it is also open in $L$.
\end{proof}

\begin{lem}\label{min_element}
Let $X$ be a nonempty closed subset of a pseudo-line $L$.
Then $X$ contains a $\leqq_L$-smallest and a $\leqq_L$-biggest element.
\end{lem}

\begin{proof}
First we show that $X$ contains a $\leqq_L$-biggest element.
Since $L$ is compact, $X$ is compact in the subspace topology.
For every $x\in X$ let $O_x$ be the set of those elements of $X$ that are 
$\leqq_L$-smaller than it but not equal to it. This set is open by \autoref{min_open}. 
Suppose for a contradiction that $X$ does not have a maximal element.
Then the $O_x$ cover $X$. Thus the $O_x$ have a finite subcover.
Since $(L,\leqq_L)$ is a linear order, the $O_x$ are nested.
Thus we may assume that the subcover consists of a single element $O_x$.
This is a contradiction since $x\notin O_x$.

The proof that $X$ contains a $\leqq_L$-smallest element is symmetric.
\end{proof}

The \emph{concatenation of two pseudo-lines $L$ and $M$}
is obtained from the disjoint union of $L$ and $M$ by identifying
the end-vertex of $L$ with the start-vertex of $M$.
\begin{rem}\label{conc_line}
The  concatenation of two pseudo-lines is a pseudo-line.
\qed
\end{rem}

\begin{rem}
Taking the concatenation of 2 pseudo-lines corresponds to taking the disjoint union of the two corresponding posets, where in the new ordering we take all elements of the second poset to be greater than all elements of the first.
\end{rem}

Let $P:L\to G$ and $Q:M\to G$ be two pseudo-arcs
such that the end-vertex $t_P$ of $P$ is the start-vertex $s_Q$ of $Q$.
Then \emph{their concatenation} is the function $f:(L\sqcup M)/(t_P=s_Q)\to G$
which restricted to $L$ is just $P$ and restricted to $M$ is just $Q$.
For a pseudo-arc $Q:M\to G$ and vertices $x$ and $y$ in the image of $Q$, we write $xQy$
for the restriction of $Q$ to those points of $M$ that are both $\leqq_L$-bigger than $Q^{-1}(x)$ and 
$\leqq_L$-smaller than $Q^{-1}(y)$. Note that $xQy$ is a pseudo-arc from $x$ to $y$. 
If $Q$ is a pseudo-arc from $v$ to $w$ and $x$ and $y$ are vertices in the image of $Q$,
we abbreviate $xQw$ by $xQ$ and $vQy$ by $Qy$.

\begin{lem}\label{con_parc}
Let $P:L\to G$ be a pseudo-arc from $x$ to $y$ and $Q:M\to G$ 
be a pseudo-arc from $y$ to $z$.
Then the concatenation of $P$ and $Q$ includes a pseudo-arc from $x$ to $z$
\end{lem}

The corresponding Lemma about arcs needs the requirement that $x\neq z$.
However, we avoid this requirement because there is a pseudo-line whose start- and end-vertex are equal, namely the trivial pseudo-line.

\begin{proof}
Let $I$ be the intersection of the image of $P$ with the image of $Q$, 
which is closed, being the intersection of two closed sets. 
Then $P^{-1}(I)$ is closed as $P$ is continuous, and contains a $\leqq_L$-minimal element $w$ by \autoref{min_element}.

If $w$ is not a vertex, then $P(w)$ is not a vertex and thus is contained in $\iota_e((0,1))$ for some edge $e$. 
Since $P$ and $Q$ both contain the whole of $\iota_e([0,1])$ if they contain some point from $\iota_e((0,1))$, the same is true for $I$.
But then $\iota_e([0,1])\subseteq I$, which contradicts the choice of $w$.
Hence $w$ is a vertex. Let $w' = P(w)$

Thus $w'Q$ is a pseudo-arc. By \autoref{conc_line}, the concatenation of $Pw'$ and $w'Q$ is the desired pseudo-arc since their images meet precisely in $w'$.
\end{proof}

A \emph{pseudo-circle} is a connected and compact graph-like space  $(G, E,V)$ with at least one edge, satisfying the following.
\begin{enumerate}
\item Removing any edge does not disconnect $G$ but removing any pair does.
\item Any two vertices can be separated by removing a pair of edges.
\end{enumerate}

We have the following relation between pseudo-lines and pseudo-circles.
\begin{lem}\label{arc_cir}
Every pseudo-circle $C$ with one edge removed is a pseudo-line
with endvertices the endvertices of the removed edge.

Conversely, let $P$ and $Q$ be pseudo-lines 
where $P$ has endvertices $s_P$ and $t_P$ and $Q$ has endvertices $s_Q$ and $t_Q$.
Then the graph-like space obtained from the disjoint union of 
$P$ and $Q$ by identifying $s_P$ with $t_Q$ and $t_P$ with $s_Q$ is a pseudo-circle or else is the trivial graph-like space.
\end{lem}

\begin{proof}
The first assertion follows immediately from the definitions.
For the second, note that (1) and (2) in the definition of a pseudo-circle are immediate 
and the graph-like space considered here is compact since the disjoint union of of two compact spaces is compact and any quotient topology of a compact space is again compact.
\end{proof}

Combining \autoref{line_contract} with \autoref{arc_cir}, we obtain the following:

\begin{cor}\label{circle_contract}
Any contraction of a pseudo-circle in which not all edges are contracted is a pseudo-circle.
\qed
\end{cor}

Using \autoref{count_line} we get:

\begin{cor}\label{count_circle}
Any countable pseudo-circle is homeomorphic to $S^1$.
\qed
\end{cor}

\begin{dfn}\label{dfn:cyc_ord}
A \emph{cyclic order on a set $X$} is a relation $R\subseteq X^3$, written $[a, b, c]_R$, that satisfies the following axioms:
\begin{enumerate}
 \item Cyclicity: If $[a, b, c]_R$ then $[b, c, a]_R$.
\item Asymmetry: If $[a, b, c]_R$ then not $[c, b, a]_R$.
\item Transitivity: If $[a, b, c]_R$ and $[a, c, d]_R$ then $[a, b, d]_R$.
\item Totality: If $a, b$, and $c$ are distinct, then either $[a, b, c]_R$ or $[c, b, a]_R$.
\end{enumerate}
\end{dfn}

Combining \autoref{pline_lin_order} and \autoref{L_is_LP} with \autoref{arc_cir}
we obtain the following.

\begin{cor}\label{cyclic_order}
The edge set of a pseudo-circle $C$ has a canonical cyclic order $R_C$ (up to choosing an orientation).
Conversely, for any nonempty cyclic order there exists a 
pseudo-circle (unique up to isomorphism) such that its edge set has the same cyclic order.
\qed
\end{cor}

We also get a cyclic order $R'_C$ on the set of all points of a pseudo-circle $C$, corresponding to the order $\leqq_L$ on the set of points of a pseudo-line $L$. Once more there are two canonical choices of cyclic order on $C$, one for each orientation of $C$; in fact, we shall take this as our definition of an orientation of $C$. For us, an orientation of a pseudo-circle $C$ is a choice of one of the two canonical cyclic orders of the points of $C$.

Let $s\subseteq o$ and let $R\subseteq o^3$ be a cyclic order.
The \emph{cyclic order of $s$ inherited from $R$} is $R$ restricted to $s^3$.
We say that $e,g$ are \emph{clockwise adjacent} in the cyclic order $R$ if $[e,g,f]_R$ for any other $f$ in $o$.
In a finite cyclic order, for each $e$ there is a unique $g$ clockwise adjacent to $e$, which we denote by $n(e)$.

Combining \autoref{components_expl_arc} with \autoref{arc_cir}
we obtain the following.

\begin{cor}\label{components_expl}
Let $s$ be a finite nonempty set of edges of a pseudo-circle $C$.
Let $S=\bigcup_{e\in s} \iota_{e}((0,1))$.
Then $L\sm S$ has $|s|$ components each of which is a pseudo-line.

For each such component there is a unique $e\in s$ such that the 
component contains precisely those edges $f$ with $[e,f,n(e)]_{R_C}$, where $n(e)$ is taken with respect to the induced cyclic order on $s$. 
\qed
\end{cor}

For a graph-like space $G$, we also use the term pseudo-circle to describe an injective map of graph-like spaces from a pseudo-circle to $G$, as well as the image of such a map. In particular, a \emph{pseudo-circle in $G$} is the image of such a map (or, in other words, it is a subspace of $G$ which is also a pseudo-circle). If $G$ is a graph-like space and $C$ is a pseudo-circle in $G$, the set of edges of $C$ is called a \emph{topological circuit} of $G$. Thus the pseudo-circles in $G$ are precisely the standard subspaces of $G$ corresponding to the topological circuits.

\begin{lem}\label{never_once}
The intersection of a topological circuit with a topological cut is never only one edge.
\end{lem}

\begin{proof}
Suppose for a contradiction that there are a topological circuit $o$ and a topological cut $b$ that intersect in only one edge $f$.
In the graph-like space $\overline o$, the set $b\cap o$ is a topological cut consisting of a single edge $f$. This contradicts the fact that removing any edge does not disconnect the pseudo-circle $\overline o$, which completes the proof.
\end{proof}

We can also show that the intersection of topological circuits with topological cuts is finite. In fact, we can prove something a little more general.

\begin{lem}\label{gls_is_tame}
Let $o$ be a set of edges in a graph-like space $G$ such that $\overline o$ is compact. The the intersection of $o$ with any topological cut $b$ is finite.
\end{lem}

\begin{proof} Let $b$ be induced by the open sets $U$ and $U'$. The sets $U \cap \overline o$ and $U' \cap \overline o$, together with all the sets $(0, 1) \times \{e\}$ with $e \in o$, comprise an open cover of $\overline o$. So there is a finite subcover, which can only contain $(0, 1) \times \{e\}$ for finitely many edges $e$. For any other edge $f$ of $o$ we must have $(0, 1)\times \{f\} \subseteq U \cup U'$, and it must be a subset either of $U$ or of $V$ since it is connected: in particular, no such $f$ can be in $b$.
\end{proof}

\section{Graph-like spaces inducing matroids}\label{sec5}

If for a graph-like space $G$ there is a matroid $M$ on $E(G)$ whose circuits are precisely the topological circuits of $G$ and whose cocircuits are precisely the topological bonds of $G$, then we say that  \emph{$G$ induces $M$}, and we may denote $M$ by $M(G)$.
Note that there can only be one such matroid since a matroid is uniquely defined by its set of circuits.

The graph-like spaces in Figures \autoref{fig:ladder} and \autoref{fig:sierpinski} both induce matroids: for the first, this is not hard to check explicitly. The second induces a matroid because it is compact: we shall show in \cite{BCC:graphic_matroids_partII} that any compact graph-like space induces a matroid. Similarly, for any finitely separable graph $G$ the space $|G|$ induces the topological cycle matroid $M_C(G)$. The one-point compactification of a locally finite graph $G$ induces the algebraic cycle matroid $M_A(G)$; if $G$ is not locally finite and does not have a subdivision of the Bean graph, a similar construction can be used to construct a noncompact graph-like space that induces $M_A(G)$.
Finally, the geometric realisation of $G$ induces the finite cycle matroid $M_{FC}(G)$.

\begin{lem}\label{minor_consistency}
Let $G$ be a graph-like space, and suppose $G$ induces a matroid $M$. Then for any $C, D\subseteq E(M)$, the graph-like space $G/C\backslash D$ induces $M/C\backslash D$. 
\end{lem}

\begin{proof} Let ${\cal C}$ and ${\cal C}^*$ be respectively the collection of topological circuits and the collection of topological cuts of $G/C\backslash D$. We will show that every circuit of $M/C\backslash D$ is in ${\cal C}$, and that every cocircuit of $M/C\backslash D$ is in ${\cal C}^*$. Lemma \ref{never_once} states that for every $o\in {\cal C}$, $b\in {\cal C}^*$, $|o\cap b|\neq 1$, so it will follow by Lemma \ref{magic_lemma} that the topological circuits of $G/C\backslash D$ are the circuits of $M/C\backslash D$ and that the minimal topological cuts (i.e. the topological bonds) of $G/C\backslash D$ are the cocircuits of $M/C\backslash D$, completing the proof.

Let $o$ be a circuit of $M/C\backslash D$. By \autoref{rest_cir} there is a circuit $o'$ of $M$ such that $o\subseteq o'\subseteq o\cup C$. Since $o'$ is a circuit of $M$, there is a pseudo-circle $O$ in $G$ with edge-set $o'$. Let $f_C: G\to G/C$ be as in the definition of the contraction $G/C$. Then $f_C\restric_O$ is a map of graph-like spaces from $O$ to a subspace of $G/C\backslash D$ that has edge-set $o$. If it describes a contraction of $O\cap C$, then Lemma \ref{circle_contract} implies that $o$ is a circuit of $G/C\backslash D$ as required. Otherwise, some vertex of $G / C \backslash D$ must contain two vertices $p$ and $q$ of $O$ such that their deletion from the pseudo-circle $O$ leaves two elements $e$ and $f$ of $o$ in different components of $O-p-q$. Then by \autoref{o_cap_b} there is a cocircuit $b$ of $M/C \backslash D$ with $o \cap b = \{e, f\}$. Using the dual of \autoref{rest_cir}, there is a cocircuit $b'$ of $M$ with $b \se b' \se b \cup D$, so that $o' \cap b' = \{e, f\}$. $b'
$ is a topological bond of $G$ not meeting $C$ and with $p$ and $q$ on opposite sides, contradicting the assumption that they are identified when we contract $C$.

Let $b$ be a cocircuit of $M/C\backslash D$. It follows by the dual of Lemma \ref{rest_cir} that there is a cocircuit $b'$ of $M$ (hence also a topological cut of $G$) such that $b\subseteq b'\subseteq b\cup D$. Let $U, V$ be the disjoint open sets in $G$ that partition $V(G)$ so that the set of edges with an end in each of $U$ and $V$ is $b'$. Let $f_C: G\mapsto G/C$ be the map of graph-like spaces describing the contraction of $C$ from $G$. Since $b'$ is disjoint from $C$, $f_C$ does not identify any element of $U$ with any element of $V$. Thus $f_C(U), f_C(V)$ are open sets in $G/C\backslash D$, and $b$ is the set of edges with an end in each, showing that $b$ is a topological cut of $G/C\backslash D$, as required. 
 
\end{proof}

\subsection{Properties of graph-like spaces inducing matroids}

Not all graph-like spaces induce matroids. In this subsection, we will consider some substructures of graph-like spaces which prevent them from inducing matroids.

   \begin{figure} [htpb]   
\begin{center}
   	  \includegraphics[height=2cm]{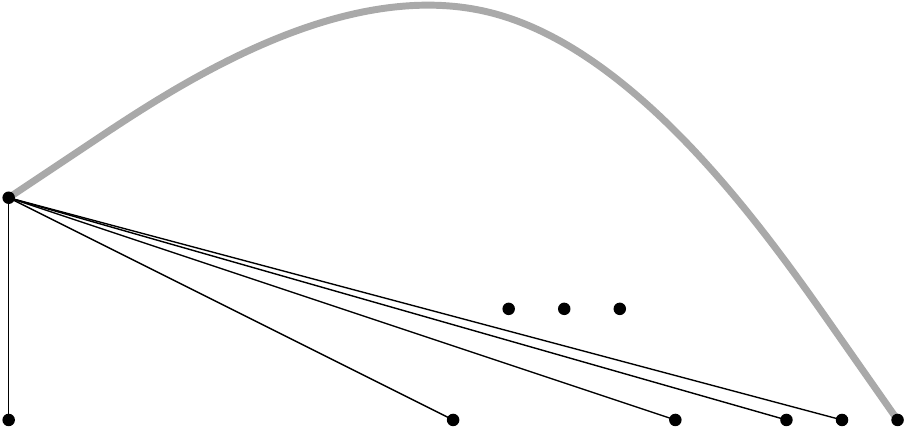}
   	  \caption{The graph-like space $F$.}
   	  \label{alm:bean}
\end{center}
   \end{figure}

The graph-like space $F$ of \autoref{alm:bean}, whose topology is that induced by the embedding in the plane suggested by the figure, does not induce a matroid.
Indeed, the curved gray edge is not a topological bond but also does not lie on a topological circuit.
Since in any matroid any edge either is a cocircuit or lies on a circuit, $F$ does not induce a matroid.

For a similar reason the graph-like space depicted in \autoref{fig_2}
does not induce a matroid.

   \begin{figure} [htpb]   
\begin{center}
   	  \includegraphics[height=3 cm]{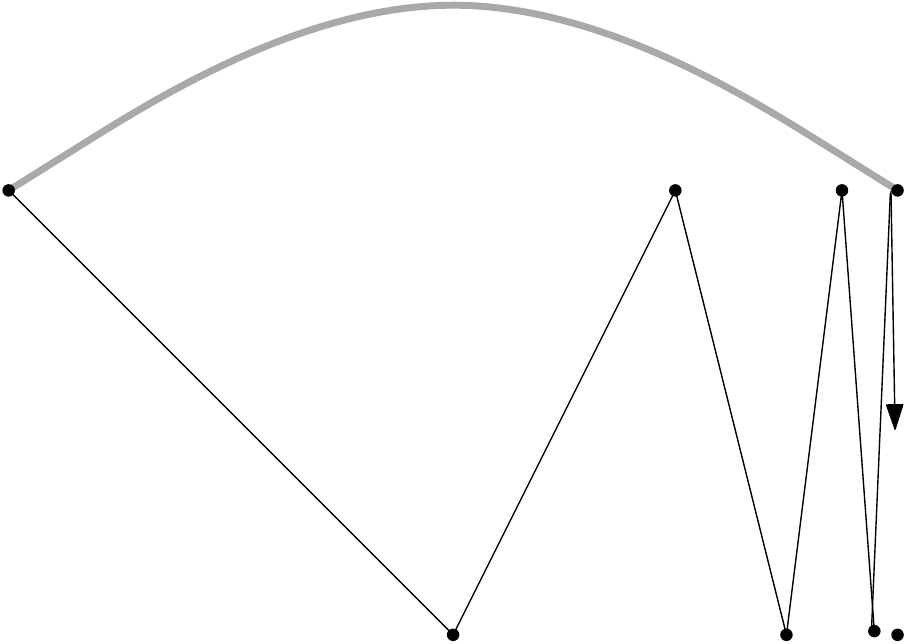}
   	  \caption{A graph-like space that does not induce a matroid.}
   	  \label{fig_2}
\end{center}
   \end{figure}

A consequence of \autoref{minor_consistency} is that no graph-like space inducing a matroid can have $F$ as a minor.
At the end of this section, we shall describe a more complex substructure that witnesses that any graph-like space including it, does not induce a matroid.

We say that a subset $H$ of a graph-like space $G$ is pseudo-arc connected if either it is a connected subset of the set of interior points of some edge of $G$ or else any two vertices in $H$ can be joined by a pseudo-arc of $G$ which is included in $H$ and any interior point of an edge $e$ of $H$ can be joined to an end-vertex of $e$ by an arc of $H \cap \iota_e``[0, 1]$.

We might hope that connected standard subspaces of graph-like spaces would necessarily be pseudo-arc connected. However, this is not true, as the standard subspace given by the black edges in \autoref{alm:bean} shows.
Even if the graph-like space induces a matroid, the property may still not hold, as illustrated in \autoref{fig:not_path_con}. However, in that example there is no pseudo-arc at all in the whole graph-like space from the vertex at the top to any other vertex.
The next lemma guarantees a pseudo-arc between two points of a connected standard subspace
provided that neither of these two obstructions occurs.

   \begin{figure} [htpb]   
\begin{center}
   	  \includegraphics[height=2cm]{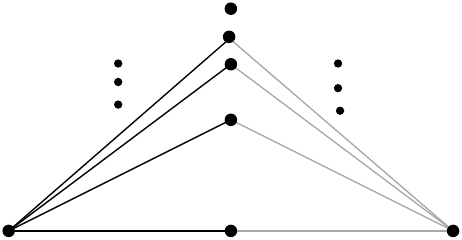}
   	  \caption{A graph-like space with a connected standard subspace, depicted in black, that is not pseudo-arc connected.}
   	  \label{fig:not_path_con}
\end{center}
   \end{figure}

\begin{lem}\label{p-arc_inside}
Let $X$ be a  connected standard subspace of a graph-like space $G$ inducing some matroid $M$. Assume there is an $X$-$X$-pseudo-arc $P$
meeting $X$ only in its two endvertices $x$ and $y$.
Then there is an $x$-$y$-pseudo-arc included in $X$. 
\end{lem}

\begin{proof}
If $P$ is a trivial pseudo-arc then $x = y$ and we are done. Otherwise, let $e$ be any edge of $P$, and let the end-vertices of $e$ be $p$ and $q$.
By the definition of contraction, $\{e\}$ becomes a topological circuit in $G/(E(X) \cup E(P) - e)$.
So by \autoref{rest_cir}, $\{e\}$ extends to an $M$-circuit $o'$ by adding contracted edges.
Let $C$ be the corresponding pseudo-circle, and let $R$ be the pseudo-arc formed by removing $e$ from $C$. Since $X$ is closed, there must be a first point $v_1$ of $R$ in $X$. Since $P$ is closed, this point also must lie on $P$, and hence it is either $x$ or $y$. Similarly, the last point $v_2$ of $R$ in $X$ is also $x$ or $y$. Since $P$ is independent, the points $v_1$ and $v_2$ cannot be equal.
Since $X$ and $P$ are closed, and just meet in two points, the arc $v_1Rv_2$ does not meet $P$ outside its endpoints. In particular, it is included in $X$. Thus it is the 
desired $x$-$y$-pseudo-arc, completing the proof.
\end{proof}

\begin{prop}\label{weak_base_arccon}
Let $G$ be a graph-like space inducing a connected matroid $M$ with a base $s$.
Then for any edges $e$ and $f$ of $M$, and any endvertices  $v$ of $e$ and $w$ of $f$, there is a unique pseudo-arc from $v$ to $w$ that uses only edges in $s$.
\end{prop}

\begin{proof}
By \autoref{swseq}, we can find a switching sequence $(e_i | 1 \leq i \leq n)$ for $s$ with first term $e$ and last term $f$. Pick a sequence $(v_i | 1 \leq i \leq n)$, with first term $v$ and last term $w$, where for each $i$ the vertex $v_i$ is an endvertex of $e_i$. Then for any $i < n$ we can find a pseudo-arc from $v_i$ to $v_{i+1}$ using only edges of $s$: if $e_i \in s$ then we take an interval of the pseudo-arc $\overline{o_{e_{i+1}}} \sm e_{i+1}$, and if $e_i \not \in s$ then we take an interval of the pseudo-arc $\overline{o_{e_i}} \sm e_i$. Repeatedly applying \autoref{con_parc} we find the desired pseudo-arc from $v$ to $w$.

To show uniqueness, we suppose for a contradiction that there are 2 distinct such pseudo-arcs $R_1$ and $R_2$. Then without loss of generality there is an edge $e_0$ in $R_1 \setminus R_2$.

Let $a\in R_1\cap R_2$ be the $\leqq_{R_1}$-smallest point that is still $\leqq_{R_1}$-bigger than any point on $e_0$; such a point exists as the intersection of the two pseudo-arcs is closed.
Similarly, let $b\in R_1\cap R_2$ be the $\leqq_{R_1}$-biggest point that is still $\leqq_{R_1}$-smaller than any point on $e_0$.
Then $aR_1b$ and $bR_2a$ are internally disjoint. Therefore $aR_1bR_2a$ is a pseudo-circle
all of whose edges are in $s$, a contradiction.
\end{proof}

\begin{rem}\label{unique_rem}
The proof of uniqueness above does not make use of the assumption that $v$ and $w$
are endvertices of edges.
\end{rem}

Let us call the pseudo-arc whose uniqueness is noted above
$vsw$ by analogy to the special case where $s$ is a pseudo-arc. Next, we give a precise description of $vsw$.

\begin{prop}\label{weak_base_arccon_2}
 The pseudo-arc $vsw$ contains precisely those edges of $s$ whose fundamental cocircuit with respect to $s$
separates $v$ from $w$.
Its linear order is given by $e\leq f$ if and only if $e$ lies on the same side as $v$ of the fundamental cocircuit $b_f$ of $f$.
\end{prop}

\begin{proof}
 Let $R$ be the pseudo-arc from $v$ to $w$ using edges in $s$ only.
Since $R$ is connected, it must contain all edges 
whose fundamental cocircuit with respect to $s$
separates $v$ from $w$.

On the other hand let $e$ be an edge on $R$.
Let $z_1$ and $z_2$ be the endvertices of $e$, with $z_1\leqq_R z_2$.
Then by the above we can join $v$ to $z_1$ by the pseudo-arc $vRz_1$ and $w$ to $z_2$
by the pseudo-arc $wRz_2$.
In $G$ with the fundamental cocircuit of $e$ removed, $z_1$ and $z_2$ lie on different sides, which we will call $A_1$ and $A_2$.
Since $vRz_1\subseteq A_1$ and $wRz_2\subseteq A_2$, 
the fundamental cocircuit of $e$ separates $v$ from $w$, which completes the proof of the first part.

The second part is immediate from the definitions.
\end{proof}

The next lemma gives a useful forbidden substructure 
for graph-like spaces inducing matroids, inspired by \autoref{alm:bean}.

\begin{lem}\label{lem_cir_eli_3_half_way}
 Let $G$ be a graph-like space, and let $v$ be a vertex in it.
Let $\{Q_{n}|n\in \Nbb\}$ be a set of pseudo-arcs starting at $v$, and vertex-disjoint apart from that. Suppose also that  the union of the edge sets of the $Q_n$ is independent.
Let $y$ be a point in the closure of the set of their endvertices. Assume there is a nontrivial $v$-$y$-pseudo-arc $P$ that is vertex-disjoint from all the $Q_n-v$.

Then $G$ does not induce a matroid.
\end{lem}

\begin{proof}

First, we shall show that $\left (\bigcup_{n\in \Nbb} Q_n\right )\cup P$ does not include a pseudo-circle. Suppose for a contradiction that
it includes a pseudo-circle $K$. Then $K$ must include some edge $e$ from $P$ and some edge $f$ from $Q_m$ for some $m\in \Nbb$. Going along $K$ starting from $f$ until we first hit 
the closed set $P$, we get two disjoint pseudo-arcs $L_1$ and $L_2$, one for each cyclic order of $K$. Formally, we consider the pseudo-arc $K-f$ endowed with the linear order $\leqq_{K-f}$. 
Let $s$ be its start vertex and $t$ be its endvertex.
Let $l_1$ be the first point of $K-f$ in $P$, and let $l_2$ be the last point of $K-f$ in $P$. Then $L_1=s(K-f)l_1$ and $L_2=l_2(K-f)t$.

We shall show that each of these pseudo-arcs contains $v$.
Since $f$ and $P-v$ are in different components of $(P\cup Q_m)-v$, each $L_i$ contains either $v$ or some edge $f'$ in some $Q_l$ with $l\in \Nbb-m$. Note that $fL_if'$ is included in $\bigcup_{n\in \Nbb} Q_n$ and is an $f$-$f'$-pseudo-arc. By the independence of $\bigcup_{n\in \Nbb} Q_n$  and \autoref{unique_rem}, it must be that
$fL_if' = fQ_mvQ_lf'$. In particular, $v\in L_i$, as desired.
This contradicts that $L_1$ and $L_2$ are disjoint.
Thus $\left (\bigcup_{n\in \Nbb} Q_n\right )\cup P$ does not include a pseudo-circle.

Now suppose for a contradiction that $G$ induces a matroid $M$.
We pick $e\in P$ arbitrarily.
Since $\left (\bigcup_{n\in \Nbb} Q_n\right )\cup P$
is $M$-independent as shown above, by \autoref{almost_fundamental_cocircuit} there must be a
cocircuit meeting $\left (\bigcup_{n\in \Nbb} Q_n\right )\cup P$
precisely in $e$.
This cocircuit defines a topological cut of $G$ with the two endvertices of $e$ on different sides. This contradicts that 
$\left (\bigcup_{n\in \Nbb} Q_n\right )\cup (P-e)$ is connected. 
\end{proof}

 \autoref{fig:forbidden_substructure}:
   \begin{figure} [htpb]   
\begin{center}
   	  \includegraphics[height=3cm]{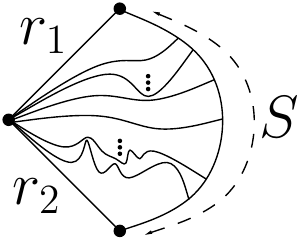}
   	  \caption{The situation of \autoref{cir_eli_3}.}
   	  \label{fig:forbidden_substructure}
\end{center}
   \end{figure}

\begin{lem}\label{cir_eli_3}
Let $G$ be a graph-like space in which there is a pseudo-circle $C$
with a vertex $v$ of $C$ that is indicent with two edges $r_1$
and $r_2$ of $C$.
Let $S$ be the pseudo-arc with edge set $E(C)-r_1-r_2$.
Assume there are infinitely many pseudo-arcs $Q_n$ starting at $v$ to points in $S$
that are vertex-disjoint aside from $v$.

If $\bigcup_{n\in \Nbb} Q_n$ does not include a pseudo-circle, 
then $G$ does not induce a matroid.
\end{lem}

\begin{proof}
Without loss of generality, we may assume that the pseudo-arcs $Q_n$ only meet $S$ in their end-vertices.
By Ramsey's theorem there is an infinite subset $N$ of $\Nbb$ such that the endpoints in $S$ of the $Q_n$ for $n\in N$ form a sequence that is either increasing or decreasing
with respect to the linear order $\leqq_{S}$ of the pseudo-arc $S$.
Let $y$ be their limit point.
Let $P$ be the $v$-$y$-pseudo-arc included in $C$ that avoids all the endpoints of those $Q_n$ with $n\in N$. Note that $P$ is nontrivial since it has to include either $r_1$ or $r_2$.
Applying \autoref{lem_cir_eli_3_half_way} now gives the desired result.
\end{proof}

\begin{cor}\label{cir_eli_single}
Let $G$ be a graph-like space, $C$ a pseudo-circle of $G$, and $r_1$ and $r_2$ distinct edges of $C$. Let $S_1$ and $S_2$ be the two components of  $C\sm \{r_1, r_2\}$. If there is an infinite set $W$ of edges of $G$ each with one end-vertex in $S_1$ and the other in $S_2$ and with all of their end-vertices in $S_2$ distinct, then $G$ does not induce a matroid.
\end{cor}
\begin{proof}
Let $G'$ be the graph-like space obtained from $G$ by contracting all edges of $S_1$. Then in $G'$, there is a vertex $v$ that is endvertex of all edges in $W$.
On the other hand, the other endvertices are distinct for any two edges in $W$.
Indeed, let $b$ be the cocircuit meeting $C$ in precisely $r_1$ and $r_2$. Then $W\se b$
and no two endvertices in $S_2$ are identified.
 
The set $\overline W$ cannot include a pseudo-circle with at least 3 edges since then $v$ would be an endvertex of at least 3 edges of that pseudo-circle, which is impossible. So by \autoref{cir_eli_3} with each of the $Q_n$ given by a single edge of $W$, we obtain that $G'$ does not induce a matroid. By \autoref{minor_consistency}, nor does $G$.
\end{proof}

We have now identified a few forbidden substructures for graph-like spaces that induce matroids. If we delete a single vertex in such a graph-like space we should not get more of these substructures. 
This motivates the following open problem:

\begin{oque}
Let $G$ be a graph-like that induces a matroid, and let $v$ be one of its vertices.
Let $G-v$ be the graph-like space obtained from $G$ by deleting $v$ and all its incident edges. Must $G-v$ induce a matroid?
\end{oque}

The main applications of these forbidden substructure results will come in Sections \ref{sec8} and \ref{sec9}. First, though, we conclude this section by giving an application which shows that the `essential' vertices of graph-like spaces inducing connected matroids, namely those vertices that lie on at least one pseudo-circle, are well behaved. We begin with some basic results about such vertices.

\begin{lem}\label{edge_of_b}
Let $G$ be a graph-like space inducing a connected matroid, let $e$ be some edge, and let $x$ be a vertex lying on some pseudo-circle $C$.
Then there is a pseudo-circle that contains $x$ and contains $e$. 
\end{lem}

\begin{proof}
If $e$ lies on $C$, we are done.
Otherwise, we pick a $G$-circuit 
$\hat C$ containing an edge of $C$ and $e$. Then $\hat C-e$ is a pseudo-arc which has a first vertex $a_1$ and a last vertex $a_2$ on $C$. Note that $a_1\neq a_2$ since $\hat C$ and $C$ have an edge in common which is not $e$. Then one of $a_1Ca_2$ or $a_2Ca_1$ contains $x$, say 
$a_1Ca_2$. Let $P$ be the $a_1$-$a_2$-pseudo-arc included in $\hat C$ that contains $e$.
Then the concatenation of $a_1Ca_2$ and $P$ is the desired circuit.
\end{proof}

With a similar argument to the one used in this proof, one can deduce the following.

\begin{cor}\label{cor8_2}
 Let $G$ be a graph-like space inducing a connected matroid and let $v$ and $w$ be vertices of $G$ each of which lies on a pseudo-circle.
Then there is a pseudo-circle through both $v$ and $w$.
\qed
\end{cor}

If $b$ is a topological bond in a graph-like space then it is possible for there to be a vertex of $\overline b$ which is not an endpoint of any edge in $b$. But none of these points can lie on any pseudo-circle of $G$, if $G$ induces a connected matroid.

\begin{lem}\label{closure_bond_circuit}
Let $G$ be a graph-like space inducing a connected matroid $M$.
Let $b$ be an $M$-bond and $x$ be a point in $\overline b$ but not an endvertex of any edge of $b$. 
Then $x$ does not lie on any $G$-pseudo-circle.
\end{lem}
\begin{proof}
Let $U$ and $V$ be disjoint open sets inducing $b$, with $x \in V$. Let $C$ be the set of edges with both endpoints in $U$.

Suppose for a contradiction that there is a topological circuit $o$ with $x \in \overline o$. By \autoref{edge_of_b}, we may assume that $o \cap b$ is nonempty. By \autoref{gls_is_tame} we know that $o \cap b$ is finite. So by \autoref{components_expl}, $\overline o \setminus (o \cap b)$ consists of finitely many pseudo-arcs, and the endpoints of these pseudo-arcs are all distinct endpoints of edges in $b$. Let $R$ be the pseudo-arc  in this collection containing $x$. By \autoref{cont_cut}, $R$ is also a pseudo-arc of $G/C$.

By \autoref{id_bond} there is a unique vertex $u$ of $G/C$ which contains all the endpoints in $U$ of the edges in $b$. 
So if the edges in $b$ whose endpoints in $V$ are the endpoints of $R$ are $e$ and $f$, then $R$ together with $e$ and $f$ gives a pseudo-circle in $G/C$ containing both $u$ and $x$. 
Let $P$ be a pseudo-arc from $x$ to $u$ in this pseudo-circle.

Now let $b'$ be the set of edges in $b$ with an endpoint other than $u$ lying on $P$. Then $b'$ is finite by \autoref{cir_eli_single}, so $x \in \overline{b \sm b'}$, and so applying \autoref{p-arc_inside} with $X = \overline{b \sm b'}$ we get that there is a $u$-$x$-pseudo-arc $Q$ included in $\overline{b  \sm b'}$. This pseudo-arc can only consist of a single edge, which would have to lie in $b'$, which is the desired contradiction.
\end{proof}

From a matroidal perspective, points not lying on pseudo-circles are dispensible.

\begin{dfn}
If $G$ is a graph-like space, we take $\hat{G}$ to be the subspace with the same edges and with the vertices of $G$ that lie on at least one edge or pseudo-circle of $G$.
\end{dfn} 

\begin{lem}\label{delete_vertices_fine}
If a graph-like space $G$ induces a matroid $M(G)$, then so does $\hat G$ and $M(\hat G) = M(G)$.
\end{lem}
\begin{proof}
We apply \autoref{magic_lemma} with $M = M(G)$, $\Ccal$ the set of topological circles in $\hat G$ and $\Ccal^*$ the set of topological cuts in $\hat G$. 
\end{proof}

It is clear from \autoref{closure_bond_circuit} that any vertex of the standard subspace $\overline b$ with $b$ a bond of $\hat G$ must be an endpoint of an edge of $G$.

\autoref{cor8_2} tell us that if $\hat G(M,\Gamma)$ induces a connected matroid, then
it is $2$-connected in the sense that
any two of its vertices can be joined by internally disjoint pseudo-arcs.

\section{Existence}\label{sec6}

Let $G$ be a graph-like space inducing a matroid $M$. Then every finite minor of $M$ is induced by a finite minor of $G$ (finite in the sense that it only has finitely many edges) by \autoref{minor_consistency}. But this finite minor must consist simply of a graph, together with a (possibly infinite) collection of spurious vertices, by \autoref{finite_gls_is_graph} applied to the closure of the set of edges. In particular, every finite minor of $M$ is graphic. We also know that $M$ has to be tame, by \autoref{gls_is_tame}. The aim of this section is to prove that these conditions are also sufficient to show that $M$ is induced by some graph-like space. More precisely, we wish to show:

\begin{thm}\label{main}
Let $M$ be a matroid. The following are equivalent.
\begin{enumerate}
\item There is a graph-like space $G$ inducing $M$. 
\item $M$ is tame and every finite minor of $M$ is the cycle matroid of some graph.
\end{enumerate}
\end{thm}

The forward implication was proved above. The rest of this section will be devoted to proving the reverse implication. The strategy is as follows: we consider an extra structure that can be placed on certain matroids, with the following properties:

\begin{itemize}
\item There is such a structure on any matroid induced by a graph-like space (in particular, there is such a structure on any finite graphic matroid).
\item Given such a structure on a matroid $M$, we can obtain a graph-like space inducing $M$.
\item The structure is finitary.
\end{itemize}

Then we proceed as follows: given a tame matroid all of whose finite minors are graphic, we obtain a graph framework on each finite minor. Then the finitariness of the structure, together with the tameness of the matroid, allows us to show by a compactness argument that there is a graph framework on the whole matroid. From this graph framework, we build the graph-like space we need.

\subsection{Graph frameworks}

A \emph{signing} for a tame matroid $M$ is a choice of functions $c_o \colon o \to \{-1, 1\}$
for each circuit $o$ of $M$ and $d_b \colon b \to \{-1, 1\}$ for each cocircuit $b$ of $M$ such that for any circuit $o$ and cocircuit $b$ we have
$$\sum_{e \in o \cap b} c_o(e)d_b(e) = 0 \, ,$$
where the sums are evaluated over $\Zbb$. The sums are all finite since $M$ is tame. 
A tame matroid is \emph{signable} if it has a signing.

Signings for finite matroids were introduced in \cite{white}, where it was shown that a finite matroid is signable if and only if it is regular, i.e. representable over any field. This result was extended to tame infinite matroids, for a suitable infinitary notion of representability, in \cite{BC:rep_matroids}. In \cite{THINSUMS} it is shown that the standard matroids associated to graphs are all signable. The construction for a graph $G$ is as follows: we begin by choosing some orientation for each edge of $G$ (equivalently, we choose some digraph whose underlying graph is $G$). We also choose a cyclic orientation of each circuit of the matroid and an orientation of each bond used as a cocircuit of the matroid. Then $c_o(e)$ is $1$ if the orientation of $e$ agrees with the orientation of $o$ and $-1$ otherwise. Similarly, $d_b(e)$ is $1$ if the orientation of $e$ agrees with that of $b$ and $-1$ otherwise. Then the terms $c_o(e)d_b(e)$ are independent of the orientation of $e$: such a term is 1 if $o$ 
traverses $b$ at 
$e$ in a forward direction, and $-1$ if $o$ traverses $b$ at $e$ in the reverse direction. Since $o$ must traverse $b$ the same number of times in each direction, all the sums in the definition evaluate to 0.

We therefore think of a signing, in a graphic context, as providing information about the cyclic orderings of the circuits and about the direction in which each edge in a given bond points relative to that bond. In order to reach the notion of a graph framework, we need to modify the notion of a signing in two ways. Firstly, we need to add some extra information specifying on which side of a bond $b$ each edge not in $b$ lies. Secondly, we need to add some conditions saying that these data induce well-behaved cyclic orderings on the circuits. 

Recall that if $s$ has a cyclic order $R$, then we say that $p,q\in s$ are {\em clockwise adjacent} in $R$
if $[p,q,g]_R$ is in the cyclic order for all $g\in s-p-q$.

\begin{dfn}\label{dfn:g_fram}
A \emph{graph framework} on a matroid $M$ consists of a signing of $M$ 
and a map $\sigma_b:E\sm b\to \{-1,1\}$ for every cocircuit $b$, which we think of as telling us which side of the bond $b$ each edge lies on, satisfying certain conditions. First, we require that these data induce a cyclic order $R_o$ for each circuit $o$ of $M$:
For distinct elements $e$, $f$ and $g$ of $M$, we take $[e,f,g]_{R_o}$ if and only if 
 both $e,f,g\in o$ and there exists a cocircuit $b$ of $M$ such that $b\cap o=\{e,f\}$ and $\sigma_b(g)=c_o(f)d_b(f)$. That is, we require that each such relation $R_o$ satisfies the axioms for a cyclic order given in \autoref{dfn:cyc_ord}. In particular, by asymmetry and totality, we require that this condition is independent from the choice of $b$: if $o$ is a circuit with distinct elements $e$, $f$ and $g$, and $b$ and $b'$ are cocircuits such that $o\cap b=o\cap b'=\{e,f\}$, then $\sigma_b(g)= c_o(f)d_b(f)$ if and only if $\sigma_{b'}(g) = c_o(f) d_{b'}(f)$.
Let $o$ be a circuit, $b$ be a cocircuit and $s$ be a finite set with $b\cap o\subseteq s\subseteq o$. Then $s\subseteq o$ inherits a cyclic order $R_o \restric_s$from $o$.
Our final conditions are as follows: for any two $p,q\in s$ clockwise adjacent in $R_o \restric_s$ we require:
\begin{enumerate}
 \item If $p,q\in b$, then $c_o(p) d_b(p)= - c_o(q) d_b(q)$.
\item If $p,q\notin b$, then $\sigma_b(p)=\sigma_b(q)$.
\item If $p\in b$ and $q\notin b$, then $c_o(p) d_b(p)= \sigma_b(q)$.
\item If $p\notin b$ and $q\in b$, then $c_o(q) d_b(q)= -\sigma_b(p)$.
\end{enumerate}
\end{dfn}

Graph frameworks behave well with respect to the taking of minors. Let $M$ be a matroid with a graph framework, and let $N = M/C\backslash D$ be a minor of $M$. For any circuit $o$ of $N$ we may choose by \autoref{rest_cir} a circuit $o'$ of $M$ with $o \se o' \se o \cup C$. This induces a function $c_{o'} \restric_o \colon o \to \{-1, 1\}$. Similarly for any cocircuit $b$ of $N$ we may choose a cocircuit $b'$ of $N$ with $b \se b' \se b \cup D$, and this induces functions $d_{b'} \restric_b \colon b \to \{-1, 1\}$ and $\sigma_{b'} \restric_{E(N) \sm b} \colon E(N) \sm b \to \{-1, 1\}$. Then these choices comprise a graph framework on $N$, with $R_o$ given by the restriction of $R_{o'}$ to $o$.

Next we show that every matroid induced by a graph-like space has a graph framework. Let $M$ be a matroid induced by a graph-like space $G$. Fix for each topological bond of $G$ a pair $(U_b, V_b)$ of disjoint open sets in $G$ inducing $b$, and fix an orientation $R'_{\overline o}$ of the pseudo-circle $\overline o$ inducing each topological circle $o$ (recall from \autoref{sec4} that an orientation of a pseudo-circle is a choice of one of the two canonical cyclic orders of the set of points). For each topological circuit $o$, let the function $c_o \colon o \to \{-1, 1\}$ send $e$ to 1 if $[\iota_e(0), \iota_e(0.5), \iota_e(1)]_{R'_{\overline o}}$, and to $-1$ otherwise. For each topological bond $d_b$, let the function $d_b \colon b \to \{-1, 1\}$ send $e$ to 1 if $\iota_e(0) \in U_e$ and to $-1$ if $\iota_e(0) \in V_e$. Finally, for each topological bond $d_b$, let the function $\sigma_b \colon E \setminus b \to \{-1, 1\}$ send $e$ to $-1$ if the end-vertices of $e$ are both in $U_b$ and to $1$ if they are 
both in $V_b$.

\begin{lem}\label{fin_fram}
The $c_o$, $d_b$ and $\sigma_b$ defined above give a graph framework on $M$.
\end{lem}
\begin{proof}
The key point will be that the cyclic ordering $R_o$ we obtain on each circuit $o$ will be that induced by the chosen orientation $R'_{\overline o}$. So let $o$ be a topological circuit of $G$. First we show that for any distinct edges $e$, $f$ and $g$ in $o$ and any topological bond $b$ with $o \cap b = \{e, f\}$ we have $\sigma_b(g)=c_o(f)d_b(f)$ if and only if $[\iota_e(0.5), \iota_f(0.5), \iota_g(0.5)]_{R_{\overline o}}$. For any edge $e \in b$ we define $\iota^b_e \colon [0, 1] \to G$ to be like $\iota_e$ but with the orientation changed to match $b$. That is, we set $\iota^b_e(r) = \iota_e(r)$ if $\iota_e(0) \in U_b$ and $\iota_e^b(r) = \iota_e(1-r)$ if $\iota_e(0) \in V_b$.

Since the pseudo-circle $\overline o$ with edge set $o$ is compact, there can only be finitely many edges in $o$ with both endpoints in $U_b$ but some interior point not in $U_b$, so by adding the interiors of those edges to $U_b$ if necessary we may assume without loss of generality that there are no such edges, and similarly we may assume that if an edge of $o$ has both endpoints in $V_b$ then all its interior points are also in $V_b$. Thus the two pseudo-arcs obtained by removing the interior points of $e$ and $f$ from $\overline o$ are both entirely contained in $U_b \cup V_b$.
 Since each of these two pseudo-arcs is connected and precisely one endvertex of $e$ is in $U_b$, we must have that one of these pseudo-arcs, which we will call $R^U$ is included in $U_b$. And the other, which we will call $R^V$, is included in $V_b$. The end-vertices of $R^U$ must be $\iota_e^b(0)$ and $\iota_f^b(0)$, and those of $R^V$ must be $\iota_e^b(1)$ and $\iota_f^b(1)$.

Suppose first of all that $\sigma_b(g) = 1$. Let $R$ be the pseudo-arc $\iota_f^b(0) f \iota_f^b(1) R^V \iota_e^b(1)$. Then $c_o(f)d_b(f) = 1$ if and only if the ordering along $R$ agrees with the orientation of $\overline o$, which happens if and only if $[\iota_f(0.5), \iota_g(0.5), \iota_e(0.5)]_{R'_{\overline o}}$, which is equivalent to $[\iota_e(0.5), \iota_f(0.5), \iota_g(0.5)]_{R'_{\overline o}}$. The case that $\sigma_b(g) = -1$ is similar. This completes the proof that  for any distinct edges $e$, $f$ and $g$ in $o$ and any topological bond $b$ with $o \cap b = \{e, f\}$ we have $\sigma_b(g)=c_o(f)d_b(f)$ if and only if $[\iota_e(0.5), \iota_f(0.5), \iota_g(0.5)]_{R'_{\overline o}}$.

In particular, the construction of \autoref{dfn:g_fram} really does induce cyclic orders on all the circuits. We now show that these cyclic orders satisfy $(1)$-$(4)$. Let $o$, $b$, $s$, $p$ and $q$ be as in \autoref{dfn:g_fram}. Without loss of generality $\overline o$ is the whole of $G$. We may also assume without loss of generality that all edges $e$ are oriented so that $c_o(e) = 1$. Since $\overline o$ is compact we may as before assume that all interior points of edges not in $s$ are in either $U_b$ or $V_b$. Thus each of the pseudo-arcs obtained by removing the interior points of the edges in $s$, as in \autoref{components_expl}, is entirely included in $U_b$ or $V_b$. Since they both lie on one of these pseudo-arcs, $\iota_p(1)$ and $\iota_q(0)$ are either both in $U_b$ or both in $V_b$. We shall deal with the case that both are in $V_b$: the other is similar. In case $(1)$, we get $d_b(p) = 1$ and $d_b(q) = -1$. In case $(2)$, we get $\sigma_b(p) = \sigma_b(q) = 1$. In case $(3)$, we get $d_b(p) = 
1$ and $\sigma_b(
q) = 1$. Finally in case $(4)$ we get $\sigma_b(p) = 1$ and $d_b(q) = -1$. Since we are assuming that $c_o(p) = c_o(q) = 1$, in each case the desired equation is satisfied. This completes the proof.
\end{proof}

Since a graph framework is a finitary structure, we can lift it from finite minors to the whole matroid.

\begin{lem}\label{fram}
Let $M$ be a tame matroid such that every finite minor is a cycle matroid of a finite graph.
Then $M$ has a graph framework.
\end{lem}

\begin{proof}
By \autoref{fin_fram} we get a graph framework on each finite minor of $M$.
We will construct a graph framework for $M$ from these graph frameworks by a compactness argument. 
Let $\Ccal$ and $\Ccal^*$ be the sets of circuits and of cocircuits of $M$. Let $H= \bigcup_{o\in \Ccal} o\times\{o\} \sqcup \bigcup_{b\in \Ccal^*} b\times\{b\}
  \sqcup \bigcup_{\tilde b\in \Ccal^*} (E\sm \tilde b)\times\{\tilde b\}  \sqcup \bigcup_{o \in \Ccal} o \times o^3$.
Endow $X=\{-1,1\}^{H}$ with the product topology.
Any element in $X$ encodes a choice of functions $c_o\colon e \mapsto  x(o, e)$ for every circuit $o$,
functions $d_b \colon e \mapsto x(b, e)$ and $\sigma_b \colon e \mapsto x(\tilde b, e)$ for every cocircuit $\tilde b$, and ternary relations $R_o = \{(e, f, g) \in o^3 | x(e, f, g) = 1\}$ for each circuit $o$.

To comprise a graph framework, these function have to satisfy several properties.
These will be encoded by the following six types of closed sets.

For any circuit $o$ and cocircuit $b$, 
let $C_{o,b}=\{x\in X| \sum_{e\in o\cap b} x(o,e) x(b,e)=0 \}$.
Note that the functions $c_o$ and $d_b$ corresponding to any $x$ in the intersection of all these closed sets will
form a signing.

Secondly, for every circuit $o$, distinct edges $e,f,g\in o$ and cocircuit $b$ 
such that $o\cap b= \{e, f\}$, let
$C_{o,b,g}=\{x\in X| x(o, e, f, g) = x(\tilde b,g)x(o,f)x(b,f)\}$.
So $x$ is in the intersection of these closed sets if and only if the cyclic orders encoded by $x$ are given as in \autoref{dfn:g_fram}.

Thirdly any circuit $o$ and distinct elements $e$, $f$, $g$ of $o$ we set $C_{o,e,f,g,{\rm Cyc}} = \{x\in X| x(o, e, f, g) = x(o,f,g,e)\}$.
Note that for any $x$ and $o$ in the intersection of all these closed sets the relation $R_o$ derived from $x$ will satisfy the Cyclicity axiom. Similarly we get sets $C_{o,e,f,g,{\rm AT}}$ encoding the Asymmetry and Totality axioms and $C_{o,e,f,g,h,{\rm Trn}}$ encoding the Transitivity axiom.

Finally, for every circuit $o$, cocircuit $b$, finite set $s$ with $o\cap b\subseteq s$, and $p, q \in s$ distinct,
let $C_{b,o,s,p,q}$ denote the set of those $x$ such that, if $p$ and $q$ are clockwise adjacent with respect to $R_o \restric_s$, then the appropriate condition of (1)-(4) from \autoref{dfn:g_fram} is satisfied.

By construction, any $x$ in the intersection of all those closed sets gives rise to a graph framework.
As $X$ has the finite intersection property, it remains to show that any finite intersection of those closed sets is nonempty.
Given a finite family of those closed sets, let $B$ and $O$ be the set of all those cocircuits and circuits, respectively, that appear in the index of these sets.
Let $F$ be the set of those edges that either appear in the index of one of those sets or are contained in some set $s$ or appear as the intersection of a circuit in $O$ and a cocircuit in $B$. As the family is finite and $M$ is tame, the sets $B$,$O$ and $F$ are finite.

By Lemma 4.6 from \cite{BC:rep_matroids} we find a finite minor $M'$ of $M$ satisfying the following.
\begin{equation*}\label{fin_min}
\begin{minipage}[c]{0.8\textwidth}
For every $M$-circuit $o\in O$ and every $M$-cocircuit $b\in B$, there are $M'$-circuits $o'$ and $M'$-cocircuits $b'$ with $o'\cap F=o\cap F$ and $b'\cap F=b\cap F$ and $o' \cap b' = o \cap b$.
\end{minipage}
\end{equation*}

By \autoref{fin_fram} $M'$ has a graph framework $((c_o'|o\in \Ccal(M')), (d_b'|b\in \Ccal^*(M')), (\sigma_b'|b\in \Ccal^*(M')))$, giving cyclic orders $R'_{o'}$ on the circuits $o'$. Now by definition any $x$ with $c_o\restric_F=c_o'\restric_F$ and $d_b\restric_F=d_b'\restric_F$ and $\sigma_b\restric_F=\sigma_b'\restric_F$ and $R_o \restric_{o'} = R_{o'}$ for $o\in O$ and $b\in B$ will lie in the intersection of all the closed sets in the finite family, as required.
This completes the proof.
\end{proof}

\subsection{From graph frameworks to graph-like spaces}

In this subsection, we prove the following lemma, which, together with \autoref{fram}, gives the reverse implication of \autoref{main}.

\begin{lem}\label{fram_gives_G-like}
Let $M$ be a tame matroid with a graph framework $\Fcal$. Then there exists
a graph-like space $G = G(M, \Fcal)$ inducing $M$.
\end{lem}

We take our notation for the graph framework as in \autoref{dfn:g_fram}.

We begin by defining $G$. The vertex set will be $V=\{-1,1\}^{\Ccal^*(M)}$, and of course the edge set will be $E(M)$. As in \autoref{def:gls}, the underlying set of the topological space $G$ will be $V\sqcup ((0,1)\times E)$.

Next we give a subbasis for the topology of  $G$.
First of all, for any open subset $U$ of $(0, 1)$ and any edge $e \in E(M)$ we take the set $U \times \{e\}$ to be open.
The other sets in the subbasis will be denoted  $U_b^i(\epsilon_b)$ 
where $i\in \{-1,1\}$, $b\in \Ccal^*(M)$ and $\epsilon_b:b\to (0,1)$.
Roughly, $U_b^1(\epsilon_b)$ should contain everything that is above $b$ and
$U_b^{-1}(\epsilon_b)$ should contain everything that is below $b$, so that 
removing the edges of $b$ from $G$ disconnects $G$. In other words, 
$G\sm (\bigcup_{e\in b}(0,1)\times \{e\})$ should be disconnected because the open sets $U_b^1(\epsilon_b)$ and $U_b^{-1}(\epsilon_b)$ should partition it (for every $\epsilon_b$).
Formally, we define $U_b^i(\epsilon_b)$ as follows.

\begin{eqnarray*}
  U_b^i(\epsilon_b)&=&
\{v\in V| v(b)=i\} \cup \bigcup_{e\in E\sm b, \sigma_b(e)=i} (0,1)\times \{e\} \\
&\cup & \bigcup_{e\in b, d_b(e)=i} (1-\epsilon_b(e),1)\times \{e\} \cup
\bigcup_{e\in b, d_b(e)=-i} (0,\epsilon_b(e))\times \{e\}
\end{eqnarray*}

To complete the definition of $G$, it remains to define the maps $\iota_e$ for every $e\in E(M)$. For each $r\in (0,1)$, we must set $\iota_e(r)=(r,e)$. For $r\in\{0,1\}$, we let:
\[
 \iota_e(0)(b)=\begin{cases}
         \sigma_b(e) & \text{if } e\notin b\\
	 -d_b(e) & \text{if } e\in b\\
         \end{cases};
 \iota_e(1)(b)=\begin{cases}
         \sigma_b(e) & \text{if } e\notin b\\
	 d_b(e) & \text{if } e\in b\\
         \end{cases}; \  \
\]

Note that $\iota_e$ is continuous and $\iota_e \restric_{(0, 1)}$ is open.
This completes the definition of $G$.
Next, we check the following.

\begin{lem}\label{check_g-like}
$G$ is a graph-like space.
\end{lem}

\begin{proof}
The only nontrivial thing to check is that for any distinct $v, v' \in V$, there are disjoint open subsets $U, U'$ of $G$ partitioning $V(G)$ and with $v \in U$ and $v' \in U'$.
Indeed, if $v\neq v'$, there is some $b\in\Ccal^*$ such that $v(b)\neq v'(b)$, and then 
for any $\epsilon_b$ with $\epsilon_b(e)\leq1/2$ for each $e\in E(M)$, the sets $U_b^1(\epsilon_b)$ and $U_b^{-1}(\epsilon_b)$
have all the necessary properties.

\end{proof}

Having proved that $G$ is a graph-like space, it remains to show that $G$ induces $M$. This will be shown in the next few lemmas.

\begin{lem}\label{cir_is_top}
 Any circuit $o$ of $M$ is a topological circuit of $G$.
 \end{lem}

The proof, though long, is simply a matter of unwinding the above definitions, and may be skipped.

\begin{proof}
By the symmetry of the construction of $G$, we may assume without loss of generality that $c_o(e) = 1$ for all $e \in o$.
The graph framework of $M$ induces a cyclic order $R_o$ on $o$. From this cyclic order we get a corresponding pseudo-circle $C$ with edge set $o$ by \autoref{cyclic_order}.
We begin by defining a map $f$ of graph-like spaces from $C$ to $G$ as follows.
First we define $f(v)$ for a vertex $v$ by specifying $f(v)(b)$ for each cocircuit $b$ of $M$.

If $b\cap o=\emptyset$, then  $(f(v))(b)=\sigma_b(e)$ for some $e\in o$.
This is independent of the choice of $e$ by condition (2) in the definition of graph frameworks.
This ensures that $f^{-1}(U_b^{i}(\epsilon_b))=C$ if $i=\sigma_b(e)$, and
$f^{-1}(U_b^{i}(\epsilon_b))=\emptyset$ if $i=-\sigma_b(e)$.

If $b\cap o=:s$ is nonempty, then $s$ is finite as $M$ is tame.
The cyclic order of $o$ induces a cyclic order on $s \cup \{v\}$: choose $p_{v,b}$ so that $p_{v,b}$ and $v$ are clockwise adjacent in this cyclic order.
We take $(f(v))(b)= d_b(p_{v,b})$.

Finally, we define the action of $f$ on interior points of edges by $f(\iota^C_e(r)) = \iota^G_e(r)$ for $r \in (0, 1)$. We may check from the definitions above that this formula also holds at $r = 0$ and $r = 1$. First we deal with the case that $r=0$. We check the formula pointwise at each cocircuit $b$ of $M$. In the case that $b \cap o = \emptyset$, we have $f(\iota^C_e(0))(b) = \sigma_b(e) = \iota^G_e(0)(b)$. Next we consider those $b$ with $e \in b$. Let $s = o \cap b$, so that $p_{\iota^C_e(0),b}$ and $e$ are clockwise adjacent in $s$. Thus $f(\iota^C_e(0))(b) = d_b(p_{\iota^C_e(0), b}) = -d_b(e) = \iota^G_e(0)(b)$ by condition (1) in the definition of graph frameworks and our assumption that $c_o(f) = 1$ for any $f \in o$. The other possibility is that $b \cap o$ is nonempty but $e \not \in b$. In this case, let $s = b \cap o + e$, so that $p_{\iota^C_e(0),b}$ and $e$ are clockwise adjacent in $s$. Thus $f(\iota^C_e(0))(b) = d_b(p_{\iota^C_e(0)}) = \sigma_b(e) = \iota^G_e(0)$ by condition (3) in the 
definition of graph frameworks and our assumption on $c_o$. The equality $f(\iota^C_e(1)) = \iota^G_e(1)$ may also be checked pointwise. The cases with $e \not \in b$ are dealt with as before, but the case $e \in b$ needs a slightly different treatment: we note that in this case $p_{\iota^C_e(1), b} = e$, so that $f(\iota^C_e(1))(b) = d_b(e) = \iota^G_e(1)$.

It is clear by definition that $f$ is injective on interior points of edges. To see that $f$ is injective on vertices, let $v$ and $w$ be vertices of $C$ such that $f(v)=f(w)$ and suppose for a contradiction that $v\neq w$.
Since $C$ is a pseudo-circle, there are two edges $e$ and $f$ in $C$ such that $v$ and $w$ lie in different components of $C\backslash\{e,f\}$. By \autoref{o_cap_b}, there is a cocircuit $b$ of $M$ with
$o\cap b=\{e,f\}$. Without loss of generality we have $e = p_{v,b}$. It follows that $f = p_{w,b}$. Since $e$ and $f$ are clockwise adjacent in the induced cyclic order on $\{e,f\}$, we have $f(v)(b) = d_b(e) = -d_b(f) = -f(w)(b)$ by condition (1) in the definition of graph frameworks and our assumption that $c_o(f) = 1$ for any $f \in o$. This is the desired contradiction. So $f$ is injective.

To see that $f$ is continuous, we consider the inverse images of subbasic open sets of $G$. It is clear that for any edge $e$ and any open subset $U$ of $(0, 1)$, $f^{-1}(\{e\} \times U) = \{e\} \times U$ is open in $C$, so it remains to check that each set of the form $f^{-1}(U^i_b(\epsilon_b))$ is open in $C$. If $b \cap o = \emptyset$ then this set is either empty or the whole of $C$. So suppose that $b \cap o \neq \emptyset$, and let $x \in f^{-1}(U^i_b(\epsilon_b))$. If $x$ is an interior point of an edge $e$ then it is clear that some open neighborhood of $x$ of the form $\{e\} \times U$ is included in $f^{-1}(U^i_b(\epsilon_b))$. 

We are left with the case that $x$ is a vertex and $s = b \cap o \neq \emptyset$. By \autoref{components_expl}, the component of $C \backslash s$ containing $x$ is the pseudo-arc $A$ consisting of all points $y$ on $C$ with $[a, y, b]_{R_C}$, together with $a$ and $b$, for some vertices $a = \iota^C_p(1)$ and $b = \iota^C_q(0)$, where for any vertex $v$ of $A$ we have $p_{v,b} = p$ and where $p$ and $q$ are clockwise adjacent in the restriction of $R_o$ to $s$. Since $f(x) \in U^i_b(\epsilon_b)$, we have $i = f(x)(b) = d_b(p)$ and so for any other vertex $v$ of $A$ we also have $f(v)(b) = d_b(p) = i$, so that $f(v) \in U^i_b(\epsilon_b)$. For any edge $e$ of $A$, applying condition (3) in the definition of graph frameworks to $p$ and $e$ in the set $s + e$ gives $\sigma_b(e) = d_b(p) = i$, so that $f''(0,1) \times e  = (0,1) \times e \se U^i_b(\epsilon_b)$. By definition, we have $(1 - \epsilon_b(p), 1) \times \{p\} \se U^i_b(\epsilon_b)$, and using condition (1) in the definition of graph frameworks we get 
$d_b(q) = -d_b(p) = -i$, so that $(0, \epsilon_b(q)) \times \{q\} \se  U^i_b(\epsilon_b)$. We have now shown that every point $y$ of $C$ with $[\iota_p^C(1 - \epsilon_b(p)), y, \iota^C_q(\epsilon_b(q))]_{R_C}$ is in  $f^{-1}(U^i_b(\epsilon_b))$. But the set of such points is open in $C$, which completes the proof of the continuity of $f$.

We have shown that the map $f$ is a map of graph-like spaces from the pseudo-circle $C$ to $G$ and that the edges in its image are exactly those in $o$, so that $o$ is a topological circuit of $G$ as required.

\end{proof}

It is clear that any cocircuit of $M$ is a topological cut of $G$, as witnessed by the sets $U_b^{-1}(\frac12)$ and $U_b^1(\frac12)$. Combining this with Lemmas \ref{cir_is_top} and \ref{never_once}, we are in a position to apply \autoref{magic_lemma} with $\Ccal$ the set of topological circuits and $\Dcal$ the set of topological cuts in $G$. The conclusion is \autoref{fram_gives_G-like}, which together with \autoref{fram} gives us \autoref{main}.

\section{Properties of the graph-like space $G(M,\Gamma)$}\label{sec8}

In this section, we shall show that $G(M,\Gamma)$ has many nice topological properties.

\begin{thm}\label{path_con}
Let $M$ be a connected matroid with a graph-framework $\Gamma$ on it, and let $s$ be an $M$-base.

Then any two vertices of $\hat G(M,\Gamma)$ can be joined by a pseudo-arc using edges from $s$ only.
\end{thm}

First, we prove that the following special case of this theorem already implies the theorem.

\begin{lem}\label{path_con_lem}
Let $M$ be a connected matroid with a graph-framework $\Gamma$ on it, and let $s$ be an $M$-base. Let $v_0$ be an endvertex of some edge $e_0\in s$. 

Then any vertex $x$ of $\hat G(M,\Gamma)$ can be joined to $v_0$ by a pseudo-arc $P$ using edges from $s$ only.
\end{lem}

\begin{proof}[Proof that \autoref{path_con_lem} implies \autoref{path_con}.]
Let $x$ and $x'$ be two vertices of $\hat G(M,\Gamma)$. 
If $s=\emptyset$, then every edge is a loop and $x=x'$.
Hence we may assume that there is some edge $e_0$ in $s$.
Let $v_0$ be an endvertex of $e_0$.
Assuming \autoref{path_con_lem}, we obtain an $x$-$v_0$-pseudo-arc and
an $x'$-$v_0$-pseudo-arc both using only edges of $s$.
The concatenation of these two pseudo-arcs includes the desired $x$-$x'$-pseudo-arc by 
\autoref{con_parc}.
\end{proof}

\begin{proof}[Proof of \autoref{path_con_lem}]

First we define the set $X$ of those edges which we expect to lie on $P$.
Let $X$ be the set of those edges $e$ in $s$ whose fundamental cocircuit $b_e$ 
has $v_0$ and $x$ on different sides (when we consider $b_e$ as a topological bond).

To show that $X$ is the edge set of some pseudo-arc, we first define a linear order $\leq$ on $X$
via $e\leq f$ if and only if $e$ is on the same side of $b_f$ as $v_0$. 
By reversing the maps $\iota_e$ if necessary, we may assume
that $\iota_e(0)$ is not separated from $v_0$ by $b_e$ for all $e\in X$.
By \autoref{weak_base_arccon} and \autoref{weak_base_arccon_2}, 
for any $f\in X$ the set $X_f:=\{e\in X|e\leq f\}$ is the edge set of a $v_0$-$\iota_f(1)$-pseudo-arc $P_f$, where $\leq_{P_f}$ is given by $\leq$.

Hence, if $e\leq f$, then $P_e$ is an initial segment of $P_f$.
This gives rise to the following definition.
Let $P$ be the map from the pseudo-line $L(X)$, where $X$ is ordered by $\leq$, to $G$
sending each point $\leqq_{L(X)} \iota_f(1)$ to its image under $P_f$, and the endpoint $t$ of $L(X)$ to $x$.

By the above, this is well-defined and injective and continuous at each point other than $t$. Thus it remains to prove that $P$ is continuous at $t$. This is clear if $X$ is empty, so we will assume from now on that it is nonempty.
From the construction of the topology of $\hat G(M,\Gamma)$, it suffices to prove the following.
\begin{equation}\label{conti}
\begin{minipage}[c]{0.8\textwidth}
For any bond $b$, there is some $y<t$ in $L(X)$ such that $P(y)Px$
is included in a single side of the bond.
\end{minipage}
\end{equation}

This is true if $P$ does not meet $b$ at all. 
Otherwise let $v_1$ be a vertex on the other side of $b$ than $x$ that is an  endvertex of an edge of $P$.
 
By the construction of $\hat G(M,\Gamma)$, the vertex $x$ lies on some $G$-circuit $C$.
We may assume that $C$ has an edge in common with $b$ by \autoref{edge_of_b} applied with some edge $e\in b$.

Let $D$ be the space obtained by removing the interior points of the (finitely many) edges lying on $C$ and in $b$ from $C$. Let $C'$ be the connected component of $D$ containing $x$. Let $W$ be the set of vertices of $C'$ that are endvertices of edges. For each $w \in W$, let $P_w = v_1sw$ be the unique pseudo-arc from $v_1$ to $w$ using only edges from $s$. Then there are only finitely many edges in both $P_w$ and $b$, so there is a $\geq_{P_w}$-maximal such edge, which we will call $\phi(w)$. We use the expression $k(w)$ to denote the endvertex of $\phi(w)$ lying on the same side of $b$ as $v_1$.

We shall show that the function $\phi$ defined as above has only a finite image. Suppose instead, for a contradiction, that we can find an infinite sequence $(w_i)$ with all the $\phi(w_i)$ distinct. 
If for any $i$ and $j$ the pseudo-arcs $k(w_i)P_{w_i}w_i$ and $k(w_j)P_{w_j}w_j$ meet, say at some vertex $v$, then $P_{w_i}v=P_{w_j}v$ by  \autoref{unique_rem}.
And it follows that $\phi(w_i)=\phi(w_j)$, and so $i=j$.
Thus for different values of $i$, the pseudo-arcs $k(w_i)P_{w_i}w_i$ are disjoint.

In order to obtain a contradiction, we shall search for a forbidden substructure
in a minor $G'$ of $G$. 
Let $F$ be the set of those edges that are on the side of $b$ that does not contain $x$.
Let $G'=G/F$, and $M'=M/F$.
Let $V_1$ be the vertex of $G'$ that contains $v_1$.
Since $C$ and $b$ have an edge in common, we can find two edges $r_1,r_2\in C\cap b$ such that $C'+r_1+r_2$ is a pseudo-circle in the minor.
Each $k(w_i)P_{w_i}w_i$ is a pseudo-arc from $V_1$ to $w_i$ by \autoref{cont_cut}

Next, we show that the union of the edge sets of $k(w_i)P_{w_i}w_i$ is independent. 
Suppose not for a contradiction so that this set includes a pseudo-circle $K$ of $G'$.
Then $K$ contains points $x$ and $y$ in distinct pseudo-arcs  $k(w_i)P_{w_i}w_i$ and $k(w_j)P_{w_j}w_j$.
It also includes a pseudo-arc $K'$ joining these two points that avoids $V_1$.
So $K'$ is also a pseudo-arc of $G$ by \autoref{cont_cut}.
Since $K'\se s$, and by \autoref{unique_rem}, we must have that $v_1\in K'$.
This is a contradiction. So the union of the 
edge sets of $k(w_i)P_{w_i}w_i$ is independent. 
Hence we may apply \autoref{cir_eli_3} in $G'$ to the $k(w_i)P_{w_i}w_i$
 and $C'+r_1+r_2$.
Since $G'$ induces a matroid by \autoref{minor_consistency}, we get a contradiction.
Hence $\phi$ has a finite image.

This defines a partition of $W$ into finitely many classes where a class consists of those points with the same image. Hence there is a class $Z$ that has $x$ in its closure.
Let $a$ be the $\phi$-value of the edges in that class.

Next, we show that $a\in X$. This is true if $a$ lies on $v_0Pv_1$. Otherwise
since $a$ lies on the pseudo-arcs $P_{z}z$ for $z\in Z$,
all the $z$ lie on the opposite side from $v_1$, so also the opposite side from $v_0$, of the fundamental cocircuit $b_a$ of $a$.
Since this side is closed, $x$ also lies on that side. So $b_a$ separates $v_0$ from
$x$, and hence $a\in X$.

Now we are in a position to prove (\ref{conti}).
For $y$ we pick the endvertex of $a$ on the same side of $b$ as $x$. 
We are to show that $P(y)Px$ is included in the side of $b$ that contains $x$.
So let $f \in X$ such that $f$ doesn't lie on the same side of $b$ as $x$.
None of the paths $yP_z$ with $z \in Z$ contains $f$, so all such $z$ lie on the same side of $b_f$ as $y$, and so $x$ also lies on the same side of $b_f$ as $y$, so $a \geq f$. Taking the contrapositive, we obtain \eqref{conti}, completing the proof.
\end{proof}

\begin{cor}\label{path_con_cor_1}
Let $M$ be a connected matroid with a graph-framework $\Gamma$ on it, and let $s$ be a set of edges of $M$.
The following are equivalent:
\begin{enumerate}
 \item $s$ is spanning in $M$.
\item $\hat G(M, \Gamma) \restric_s$ is connected.
\item $\hat G(M, \Gamma) \restric_s$ is pseudo-arc connected.
\end{enumerate}
\end{cor}

\begin{proof}
By the above theorem, 1 implies 3 which clearly implies 2. If $\hat G(M, \Gamma) \restric_s$
is connected, then $s$ meets every topological bond of $\hat G(M, \Gamma)$, and hence is spanning in $M$.
\end{proof}

\begin{cor}\label{path_con_cor_2}
Let $M$ be a connected matroid with a graph-framework $\Gamma$ on it, and let $s$ be a set of edges of $M$.
The following are equivalent:
\begin{enumerate}
 \item $s$ is an $M$-base.
\item $s$ is minimal with the property that $\hat G(M, \Gamma) \restric_s$ is connected.
\item $s$ is minimal with the property that $\hat G(M, \Gamma) \restric s$ is pseudo-arc connected.
\end{enumerate}
\end{cor}

\begin{proof}
This follows from the last corollary and the fact that bases are minimal spanning sets.
\end{proof}

\begin{lem}\label{cut_bond}
Let $G$ be a graph-like space inducing a matroid.
Then every topological cut $t$ is a disjoint union of topological bonds.
\end{lem}

\begin{proof}
By \autoref{gls_is_tame}, every circuit meets $t$ finitely. In fact, every circuit must meet $t$ in an even number of edges, by \autoref{components_expl} and the fact that the components  mentioned in that Corollary must lie on alternate sides of $t$.

Applying this in all the cases where $t$ is a topological bond shows that the induced matroid is binary, so we may apply \autoref{binary_x} to obtain the desired result.
\end{proof}

\begin{thm}\label{standard_subspace_conny}
Let $M$ be a connected matroid with a graph-framework $\Gamma$ on it.
For any standard subspace $\bar X$ of $\hat G(M,\Gamma)$ the following are equivalent.
\begin{enumerate}
 \item $\bar X$ is connected.
\item $\bar X$ is pseudo-arc connected.
\item $\bar X$ contains an edge
of every topological cut of which it meets both sides. 
\item $\bar X$ contains an edge
of every topological bond of which it meets both sides.
\end{enumerate}
\end{thm}

\begin{proof}
Clearly 2 implies 1 and 1 implies 3. To see that 3 implies 2, 
let $s_0$ be a base of $M$ restricted to the edges of $\bar X$.
Extend $s_0$ to a base $s$ of $M$.
By \autoref{path_con} any two points in $\bar X$ are joined by a pseudo-arc using only edges of $s$.
When we contract $s_0$, then the edge set of this pseudo-arc remains independent, 
and its endvertices are identified by the definition of contraction. 
Hence this pseudo-arc does not contain any edge outside $s_0$.

The conditions 3 and 4 are equivalent: clearly 3 implies 4, and the converse follows from  \autoref{cut_bond}. This completes the proof
\end{proof}

Similar to the notion of local path connectedness is that of local pseudo-arc connectedness. A graph-like space $G$ is \emph{locally pseudo-arc connected} if every
open neighborhood $U$ of any vertex $x$ of $G$ includes an open neighborhood $U'$ of $x$ which is pseudo-arc connected in $G$. Note that if we put on $U$ the additional restriction that it is a basic open neighbourhood, then this would not change the meaning of this definition.

\begin{thm}\label{locally pseudo-arc}
Let $M$ be a connected matroid with a graph-framework $\Gamma$ on it.
Let $X$ be a set of edges such that the standard subspace $\bar X$ of
$\hat G(M,\Gamma)$ is pseudo-arc connected.

Then $\bar X$ is locally pseudo-arc connected.
\end{thm}

\begin{proof}
We must show that for each $x\in \bar X$ and each basic open set $U$ of $\hat G(M, \Gamma)$ containing $x$ there is an open set $U'$ of $\hat G(M, \Gamma)$ containing $x$ and such that $U' \cap \bar X \subseteq U \cap \bar X$ and $U' \cap \bar X$ is pseudo-arc connected in $\bar X$. If $x$ is an interior point of an edge $e$, it suffices to take $U' = U \cap \iota_e``(0, 1)$. So we may assume that $x$ is a vertex. 

Since $U$ is a basic open set, there are finitely many bonds
$b_1,\ldots, b_n$ such that $U=\bigcap_{m=1}^n U_{b_m}^{i_m}(\epsilon_m)$ for some $i_m$ and $\epsilon_m$. Let $V= \bigcup_{m=1}^n U_{b_m}^{1-i_m}(1-\epsilon_m)$.
Then $U$ and $V$ are disjoint and partition the vertices.
We can now simplify the situation a little by moving to a carefully chosen minor.

Let $P$ be the set of edges with both endvertices in $V$, 
and let $M'=(M/P)\restric_{X\sm P}$ and let  $G'=(\hat G(M,\Gamma)/P)\restric_{X\sm P}$. 
By \autoref{cont_cut}, no vertex of $U$ gets identified with any other vertex. 
Let $V'$ be the set of vertices of $G'$
that have at least one vertex of $V$ in their contraction-equivalence class.
Note that in $V'$ there are at most $n$ vertices that are endpoints of edges, by \autoref{id_bond}.
Then $V'$ and $U$ partition $V(G')$.

Now let $D$ be the set of edges of $G'$ with one endvertex in $V'$.
So $D$ is a topological cut of $G'$.
Let $Y$ be the set of vertices of $G'$ to which there is some pseudo-arc from $x$ in $G'\sm D$ containing $x$, and let $D'$
be the set of edges in $D$ with one endpoint in $Y$. We now break $D'$ up into finitely many bonds, which we can use to construct the basic open set $U'$.

\begin{sublem}\label{local_locally}
 $D'$ is a disjoint union of finitely many cocircuits $d_1\ldots d_k$ of $M'$.
\end{sublem}

Before proving \autoref{local_locally}, we explain how to use it to build a set $U'$ with the desired properties.

By the dual of \autoref{rest_cir}, each cocircuit $d_i$ extends to a cocircuit $d_i'$ of $M$ by using additionally only edges from
$E\sm (P \cup X)$.
We set
$U'=\bigcap_{m=1}^k U_{d_m'}^{x(d_m')}(\epsilon'_m)$.
Here the $\epsilon'_m(e)$ have to be chosen small enough such that
 $\iota_e((0,1))\cap U_x\se \iota_e((0,1))\cap U$. Let us take for
$\epsilon'_m(e)$ the minimum of the values $\epsilon_m(e)$ and $1-\epsilon_m(e)$ over all $m$ such that $e\in b_m$.

Next we will show that the set of vertices in $U'\cap \bar X$ is precisely $Y$. By the construction of $U$ and $U'$ as basic open sets, and by the construction of $Y$, this will suffice to show both that $U' \cap \bar X \subseteq U \cap \bar X$ and that $U' \cap \bar X$ is pseudo-arc connected in $\bar X$, completing the proof.

We will denote the set of vertices of $U' \cap \bar X$ by $Z$.

First let $y\in Y$. By assumption there is an $x$-$y$-pseudo-arc in $G'$ that avoids $D$.
Since none of its vertices got identified when constructing $G'$ from $G$, this pseudo-arc is also a pseudo-arc of $G$. Since it contains $x$ and avoids each bond $d'_m$, it must be included in $U'$, yielding that $y\in U'$, and so $y \in Z$.

Now let $y \in Z$, and suppose for a contradiction that $y \not \in Y$. 
Since $y\in \bar X$ and $\bar X$ is pseudo-arc connected, 
the vertices $x$ and $y$ can be joined by a pseudo-arc $R$ in $\hat G(M,\Gamma)$ using edges from $X$ only. This pseudo-arc must meet $D'$, but since by \autoref{gls_is_tame} it meets $D$ only finitely, it can only meet $D'$ finitely. Let $e$ be the last edge on $R$ which lies in $D'$, and let $a$ be the endpoint of $e$ which is further along $R$.

Then $a$ and $y$ can be linked by the pseudo-arc $aR$, which avoids each $d_m'$ and contains $y$, and so is a subset of $U'$. Thus $a \in U'$. By the construction of $D'$, this implies that $a \in Y$, and so by the construction of $Y$ and \autoref{con_parc} we get $y \in Y$, which is the desired contradiction. Thus $Z = Y$, as required.

This completes the proof of \autoref{locally pseudo-arc}, assuming \autoref{local_locally}.
It remains to prove \autoref{local_locally}.

\begin{proof}[Proof of \autoref{local_locally}]
Note that every circuit of $M'$ meets $D'$ finitely since it is included in the topological cut $D$.
First we show that every circuit $o$ meets $D'$ evenly. By \autoref{components_expl},
$\bar o\sm D$ is a disjoint union of finitely many pseudo-arcs. Some of these are completely included in $U$, the others in $V'$. For each such pseudo-arc $P$ included in $U$ there are precisely two edges in $o\cap D$ which meet $P$, and each edge of $D$ is in precisely one
such pair. If $P$ meets $Y$, then every vertex of $P$ must be in $Y$.
So for each $P$ either both edges are in $D'$ or neither is. So $o \cap D'$ is a disjoint union of such pairs, and so has even size.

By \autoref{binary_x}, $D'$ is a disjoint union of $M'$-cocircuits $d_m$:
$D'=\dot \bigcup_{m\in W} d_m$.
We must show that $W$ is finite.
For each $d_m$ we pick an edge $e_m \in d_m$. Let $v_m$ be the endvertex of $e_m$ that is in $V'$.

If there are $m$ and $n$ such that $v_m=v_n$, then we can join the other endvertices of the edges $e_m$ and $e_n$, which are both in $Y$,  by a pseudo-arc using no edges of $D'$. Together with $e_m$ and $e_n$ this gives a pseudo-circle. Since this pseudo-circle does not meet $d_m$ just once by \autoref{never_once}, we must have $e_n\in d_m$. Since the $d_m$ are disjoint, we get that $d_m=b_n$, and thus $m=n$.
This defines an injective map from $W$ into the finite set of vertices of $V'$ that are endvertices of edges.
Thus $W$ is finite, which completes the proof.
\end{proof}

\end{proof}

\begin{cor}\label{loc_con}
Let $M$ be a connected matroid with a graph-framework $\Gamma$ on it.
Then $\hat G(M,\Gamma)$ is locally pseudo-arc connected.
\qed
\end{cor}

The last part of this subsection is concerned with separability properties
of $G(M,\Gamma)$ and $\hat G(M,\Gamma)$. First we recall some basic definitions:
A topological space is \emph{regular (or $T3$)} if for any closed set $A$ and any 
$a\not\in A$ there exist disjoint open sets $O_1$ and $O_2$
such that $A\se O_1$ and $a\in O_2$.
A topological space is \emph{normal (or $T4$)} if for any two disjoint
closed sets $A_1$ and $A_2$ there exist disjoint open sets $O_1$ and $O_2$
such that $A_i\subseteq O_i$ for $i=1,2$. Note that every Hausdorff space that is normal is regular, and that any subspace of a regular space is regular. We shall rely on the standard fact that any compact topological space is normal.

\begin{prop}\label{normal} 
$G(M,\Gamma)$ is normal.
\end{prop}

\begin{proof}
First, we show that the subspace $V$ consisting of the vertices of $G(M, \Gamma)$ is normal. Recall that $V$ is $\{-1, 1\}^{\Ccal^*(M)}$. It is immediate from the definition of the topology of $G(M, \Gamma)$ that the subspace topology on $V$ is simply the product topology induced from the discrete topology on each factor $\{-1, 1\}$. So by Tychonov's Theorem, $V$ is compact, and therefore normal.

Now let $A_1$ and $A_2$ be disjoint closed subsets of $G(M, \Gamma)$. By normality of $V$, we can find open sets $U_1$ and $U_2$ such that $A_i \cap V \subseteq U_i$ for $i = 1,2$ and such that $U_1 \cap U_2 \cap V = \emptyset$. For any element $x$ of $A_1$, we now define an open set $O_{1,x}$ as follows:

If $x$ is an interior point of an edge $A_1$, say $x = (r,e)$, then let $\epsilon_{1, x}> 0$ be such that $(r - \epsilon_{1,x}, r + \epsilon_{1,x}) \subseteq \iota_e^{-1}(E \setminus A_2)$, which is possible since $A_2$ is closed, and let $O_{1,x} = \iota_e``(r - \frac{\epsilon_{1,x}}2, r + \frac{\epsilon_{1,x}}2)$. If $x$ is a vertex then pick some basic open neighbourhood $U_{1,x} = \bigcap_{m = 1}^{n_{1,x}} U^{x(d^m_{1,x})}_{d^m_{1,x}}(\epsilon^m_{1,x})$ of $x$ with $U_{1,x} \subseteq U_1 \setminus A_2$, and let $O_{1,x} = \bigcap_{m = 1}^n U^{x(d^m_{1,x})}_{d^m_{1,x}}\left(\frac{\epsilon^m_{1,x}}2\right)$.

We define sets $O_{2,y}$ for each $y \in A_2$ similarly. The key fact we need is that, for any $x \in A_1$ and $y \in A_2$, we have $O_{1,x} \cap O_{2, y} = \emptyset$. Suppose not for a contradiction. If $x$ and $y$ are both interior points of edges, then they must be interior points of the same edge, say $x = (r, e)$ and $y = (s, e)$, and we must have $(r - \frac{\epsilon_{1,x}}2, r + \frac{\epsilon_{1,x}}2) \cap (s - \frac{\epsilon_{2,y}}2, s + \frac{\epsilon_{2,y}}2) \neq \emptyset$, so that $|r - s| < \frac{\epsilon_{1,x} + \epsilon_{2,y}}2$. On the other hand, since $y \not \in \iota_e``(r - \epsilon_{1,x}, r + \epsilon_{1,x})$, we have $|r - s| \geq \epsilon_{1,x}$ and similarly $|r-s| \geq \epsilon_{2,y}$, so that $|r-s| \geq \frac{\epsilon_{1,x} + \epsilon_{2,y}}2$, giving the desired contradiction in this case.

If $x$ is an interior point of an edge, say $x = (r, e)$, and $y$ is a vertex, then let $p = (s, e)$ be any point contained in $O_{1,x} \cap O_{2,y}$. Since $x \not \in U_{2,y}$ but $p \in O_{2,y}$, there must be some $m \leq n_{2,y}$ such that $e \in d^m_{2,y}$: pick such an $m$ with $\epsilon^m_{2,y}$ minimal. Without loss of generality we may assume that $y(d^m_{2,y}) = -1$. By construction, $s < \frac{\epsilon^m_{2,y}}2$ and $|r-s| < \frac{\epsilon_{1,x}}2$, so that $r < \frac{\epsilon_{1,x} + \epsilon_{2,y}}2$. On the other hand, since $x \not \in U_{2,y}$ we have $r \geq \epsilon^m_{2,y}$, and since also $r \geq \epsilon_{1,x}$ we have $r \geq \frac{\epsilon_{1,x} + \epsilon_{2,y}}2$, giving the desired contradiction in this case. The case that $x$ is a vertex and $y$ is an interior point of an edge is similar.

If $x$ and $y$ are both vertices, let $p$ be any point in $O_{1,x} \cap O_{2,y}$. If $p$ were a vertex, it would have to be in both $U_1$ and $U_2$, which is impossible, so it must be an interior point of an edge, say $e$. Then at least one end-vertex of $e$ lies in each of $U_1$ and $U_2$, so $e$ has precisely one end-vertex in each of these sets. Without loss of generality, $\iota_e(0) \in U_1$ and $\iota_e(1) \in U_2$. Now there must be some $m_1 \leq n_{1,x}$ with $e \in d^m_{1,x}$ and $x(d^m_{1, x}) = -1$ and some $m_2 \leq n_{2,y}$ with $e \in d^m_{2,y}$ and $y(d^m_{2,y}) = 1$. Thus $p \in U^{x(d^{m_1}_{1,x})}_{d^{m_1}_{1,x}}(\frac{\epsilon^{m_1}_{1,x}}2) \cap U^{y(d^{m_2}_{2,y})}_{d^{m_2}_{2,y}}(\frac{\epsilon^{m_2}_{2,y}}2) \cap \iota_e``(0, 1) = \emptyset$, which is the desired contradiction in this case.

Having dealt with all possible cases, we may deduce that $O_{1, x} \cap O_{2, y} = \emptyset$ for each $x \in A_1$ and each $y \in A_2$. Now let $O_1 = \bigcup_{x \in A_1} O_{1, x}$ and $O_2 = \bigcup_{y \in A_2} O_{2,y}$. Then the $O_i$ are open, we have $A_i \subseteq O_i$ for $i = 1,2$ and $O_1 \cap O_2 = \bigcup_{x \in A_1} \bigcup_{y \in A_2} O_{1,x} \cap O_{2, y} = \emptyset$.

\end{proof}

\begin{cor}\label{reg}
 $\hat G(M,\Gamma)$ is regular.
\end{cor}

\begin{proof}
This follows from the fact that singletons are closed as 
$\hat G(M,\Gamma)$ is Hausdorff, and the fact that $G(M,\Gamma)$ is regular.
\end{proof}

With constructions like that of example 87 in \cite{counterexamples_in_topology}, it is possible to build a subspace of $G(M,\Gamma)$ that is not normal.
However, we do not know whether $\hat G(M,\Gamma)$ is normal.

\section{Countability of circuits in the 3-connected case}\label{sec9}

Our aim in this section is to prove the following:
\begin{thm}\label{countcir}
Any topological circuit in a graph-like space inducing a 3-connected matroid is countable.
\end{thm}

For the remainder of the section we fix such a graph-like space $G$, inducing a 3-connected matroid $M$, and we also fix a pseudo-circle $C$ of $G$, whose edge set gives a circuit $o$ of $M$.

We begin by taking a base $s$ of $M/o$, and letting $G' = G/s$. Thus by \autoref{minor_consistency} $G'$ induces the matroid $M' = M/s$ in which $o$ is a spanning circuit. For any $e \in o$, $o - e$ is a base of $o$ and so $s \cup o - e$ is a base of $M$, which we shall denote $s^e$. We shall call the edges of $E(M') \sm o$ which are not loops {\em bridges}. We denote the set of bridges by $\Br$. The endpoints of each bridge lie on the pseudo-circle $C'$ corresponding to $o$ in $G'$. The edges of $C'$ are the same as those of $C$, but the vertices are different: recall that the vertices of the contraction $G' = G/s$ were defined to be equivalence classes of vertices of $G$. Each of these can contain at most one vertex of $C$, since $o$ is a circuit of $M'$. Thus each vertex of $C'$ contains a unique vertex of $C$.

\begin{lem}\label{bridgestruc}
Let $g \in o$ and let $f$ be a bridge with endpoints $v'$ and $w'$ in $G'$. Let $v$ be the vertex of $C$ contained in $v'$, and $w$ the vertex of $C$ contained in $w'$. Let $x$ be the endvertex of $f$ in $G$ contained in $v'$, and $y$ the endvertex of $f$ in $G$ in contained in $w'$. Then the fundamental circuit $o_f$ of $f$ with respect to the base $s^g$ of $M$ is given by concatenating 4 pseudo-arcs: the first, from $x$ to $y$, consists of only $f$. The second, from $y$ to $w$, contains only edges of $s$. The third, from $w$ to $v$ contains only edges of $o$ - it is the interval of $C - g$ from $w$ to $v$. The fourth, from $v$ to $x$, contains only edges of $s$.
\end{lem}
\begin{proof}
$o_f \cap o$ must consist of the fundamental circuit of $f$ with respect to the base $o - g$ of $M'$ - that is, of the interval of $C' - g$ from $w'$ to $v'$. 
So the pseudo-arc $v(C-g)w$, which is the closure of this set of edges, lies on the pseudo-circle $\bar o_f$.
So $(\bar o_f-f)\sm v(C-g)w$ consists of two pseudo-arcs joining $v$ and $w$ to $x$ and $y$.
These two pseudo-arcs use edges from $s$ only. Since $v$ and $y$ lie in different connected components of $G\restric_s$, we must have that the first goes from $v$ to $x$, and the second goes from $w$ to $y$. This completes the proof.
\end{proof}

\begin{lem}\label{sepbridge}
For any distinct edges $e$ and $f$ of $C$, there is a bridge whose endvertices separate $e$ from $f$ in $C$.
\end{lem}
\begin{proof}
Since $M$ is 3-connected, $\{e,f\}$ is not a bond of $M$, so we can pick some $g\not \in \{e,f\}$ in the fundamental bond of $f$ with respect to the base $s^e$. Then $f$ lies in the fundamental circuit $o_g$ of $g$, which is therefore not a subset of $s + g$. Thus $g$ is a bridge, and since the fundamental circuit of $g$ with respect to the base $o - e$ of $M'$ contains $f$ but not $e$ the endpoints of $g$ separate $e$ from $f$. 
\end{proof}

Given that we are aiming to prove \autoref{countcir}, we may as well assume that $o$ has at least $2$ elements, and by \autoref{sepbridge} we obtain that there is at least one bridge. We now fix a particular bridge $e_0$, and make use of the 3-connectedness of $M$ to build a tree structure capturing the way the endpoints of the bridges divide up $C'$. We will call this tree the \emph{partition tree}, and define it in terms of certain auxiliary sequences $(I_n \subseteq \Br)$, $(J_n  \subseteq V(C'))$ and $(K_n)$ indexed by natural numbers, given recursively as follows:

We always construct $J_n$ from $I_n$ as the set of endvertices of elements of $I_n$, and $K_n$ as the set of components of $C' \setminus J_n$.
We take $I_0$ to be $\{e_0\}$, and $I_{n+1}$ to be the set of bridges that have endvertices in different elements of $K_n$ or at least one endvertex in $J_n$.

Then the nodes of the tree at depth $n$ will be the elements of $K_n$, with $p$ a child of $q$ if and only if it is a subset of $q$.

\begin{lem}\label{every_bridge}
 Every bridge is in some $I_n$.
\end{lem}

\begin{proof}
Suppose not, for a contradiction, and let $e$ be any bridge which is in no $I_n$. In particular, the endpoints of $e$ both lie in the same component of $C - J_0$, so there is a pseudo-arc joining them in $C$ that meets neither endvertex of $e_0$. Let $f$ be any edge of this pseudo-arc. 
Let $v_0'$ be any endvertex of $e_0$, and let $v_0$ be the unique vertex of $C$ contained in $v_0'$.

For each $n$, let $B_n$ be the element of $K_n$ of which $f$ is an edge, and let $B=\bigcap_{n \in \Nbb} B_n$ and $A = C \setminus B$. 
Note that any 2 vertices in $B$ are joined by a unique pseudo-arc in $B$, and that $A$ has the same property. Since the two endvertices of $e_0$ (in $G'$) avoid $B_1$, they are both in $A$. Since $e$ is in no $I_n$, its two endvertices lie in $B$. 

Let $A_V$ be the set of endvertices $v$ of edges of $G$ such that the first point of $vs^fv_0$ on $C$ is contained in a vertex in $A$. Let $A_E$ be the set of edges of $G$ that have both endvertices in $A_V$, and let $B_E = E(M) \sm A_E$. 
Note that for any vertex $v\in A_V$, all edges of the unique $v$-$C$-path included in $s_f$ lie in $A_E$. And for any $v\not\in A_V$, all edges of the unique $v$-$C$-path included in $s_f$ lie in $B_E$.

We shall show that $(A_E,B_E)$ is a 2-separation of $M$, which will give the desired contradiction since we are assuming that $M$ is 3-connected.

First, we show that $s^f \cap A_E$ is a base of $A_E$. It is clearly independent. Let $g$ be any edge in $A_E \sm s^f$. Suppose first of all that $g$ is a bridge. We decompose the fundamental circuit of $g$ as in \autoref{bridgestruc}, taking the notation from that lemma. Then since each of the endpoints $x$ and $y$ of $g$ is in $A_V$, every edge of this fundamental circuit is in $A_E$, as required. 

So suppose instead that $g$ isn't a bridge, that is, $g$ is a loop in $M'$. Let $R_1$ and $R_2$ be the pseudo-arcs from the endpoints $x$ and $y$ of $g$ to $v_0$ which use only edges from $s^f$. Let $z$ be the first point of $R_1$ to lie on $R_2$. Then $zR_1v_0$ and $zR_2v_0$ must be identical, as both are pseudo-arcs from $z$ to $v_0$ using only edges of $s^f$. Let $k$ be the first point on this pseudo-arc that is in $C$. By assumption, $k \in A$. Also, $xR_1zR_2y$ is a pseudo-arc from $x$ to $y$ using only edges from $s^f$, so must form (with $g$) the fundamental circuit of $g$ with respect to $s^f$, so can meet $C$ at most in a single vertex (
since $g$ is a loop in $M'$). Thus all edges in this fundamental circuit lie on either $xR_1k$ or $yR_2k$, and so are in $A_E$, as required.

Next, we show that $(s_f \cap B_E) + f$ is a base of $B_E$. It is independent since $A$ includes some edge as $e_0$ is a bridge. Let $g$ be any edge in  $B_E \sm s^f - f$. If $g$ isn't a bridge we can proceed as before, so we suppose it is a bridge. We decompose the fundamental circuit of $g$ as in \autoref{bridgestruc}, taking the notation from that Lemma. At least one of $v'$ and $w'$ lies in $B$: without loss of generality it is $v'$. Suppose for a contradiction that $w'$ is in $A$. Then either $w'$ is in some $J_n$ or it is an element of some $K_n$ not containing $f$. In either case, $g \in I_{n+1}$ and so $v' \in J_{n+1}$, giving the desired contradiction since we are assuming $v' \in B$. Thus $w'$ is also in $B$. Let $R$ be the pseudo-arc from $v$ to $w$ in $B$. Then $g$ is spanned by the pseudo-arc $xs^fvRws^fy$, which uses only edges of $s_f \cap B_E + f$.
To see this we apply \autoref{bridgestruc} with some edge not in $B_1$ in place of $f$ of that lemma.

Since each of $A_E$ and $B_E$ has at least 2 elements, and the union of the bases for them given above only contains one more element than the base $s^f$ of $M$, this gives a 2-separation of $M$, completing the proof.
\end{proof}

\begin{lem}\label{few_sons}
 Every node of the Partition-tree has at most countably many children.
\end{lem}

\begin{proof}
Let $x\in K_n$  be a node of the Partition-tree. Then the closure $\bar x$ of the set of interior points of edges of $x$ is a pseudo-arc. Let $\hat x$ be the set obtained from this pseudo-arc by removing its end-vertices.
An $x$-bridge is a bridge with one endvertex in $\hat x$ and one in its complement. Thus every element of $J_{n+1} \cap x$ must be an endvertex of an $x$-bridge or of $\bar x$.

Let $v_1$ and $v_2$ be vertices of $\hat x$ with $v_1 \leqq_{\bar x} v_2$. Suppose for a contradiction that there are infinitely many elements of $J_{n+1}$ between $v_1$ and $v_2$. Pick a corresponding set $W$ of infinitely many $x$-bridges with different attachment points between $v_1$ and $v_2$. 
Since neither of $v_1$ and $v_2$ is an endpoint of $\bar x$, there are edges $e_1$ and $e_2$ in $x$
such that all points of $e_1$ are $\leqq_{\bar x}$-smaller than $v_1$, and similarly 
all points of $e_2$ are $\leqq_{\bar x}$-bigger than $v_2$.
Then by \autoref{cir_eli_single} with $r_1=e_1$ and $r_2=e_2$, $G'$ does not induce a matroid, which gives the desired contradiction.

We have established that between any two elements of $J_{n+1}  \cap \hat x$ there are only finitely many others. Hence $J_{n+1} \cap \hat x$ is finite or has the order type of $\Nbb$, $-\Nbb$
or $\Zbb$. In all these cases there are only countably many children of $x$, since these children are the connected components of $x \sm (J_{n+1} \cap x)$.
\end{proof}

We now consider rays in the partition tree: a {\em ray} consists of a sequence $(k_n \in K_n | n \in \Nbb)$ such that for each $n$ the node $k_{n+1}$ is a child of $k_n$. Given such a ray, we call the set $\bigcap_{n \in \Nbb} k_n$ its {\em partition class}.

\begin{lem}\label{part_leq_1}
 The partition class of any ray includes at most one edge.
\end{lem}

\begin{proof}
Suppose for a contradiction that there is some ray $(k_n)$ whose partition class includes 2 different edges $e$ and $f$. Then by \autoref{sepbridge} there is a bridge $g$ whose endvertices separate $e$ from $f$ in $C$. By \autoref{every_bridge}, $g$ lies in some $I_n$. But then $e$ and $f$ lie in different elements of $K_n$, so can't both lie in $k_n$, which is the desired contradiction.
\end{proof}

For any element $k$ of $K_n$ with $n \geq 1$, the {\em parent} $p(k)$ is the unique element of $K_{n-1}$
including $k$.

An element $k$ of $K_n$ with $n \geq 2$ is \emph{good}
if no bridge in $I_n$ has endvertices in two different components of $\overline{p(p(k))\sm k}$.
Note that $\overline{p(p(k))\sm k}$ has at most two components.
Note that if $k$ is not good, there have to be two vertices in different components of 
not only $\overline{p(p(k))\sm k}$ but also $p(p(k))\sm k$. 

\begin{lem}\label{good_son}
Every node of the Partition-tree has at most one good child.
\end{lem}

\begin{proof}
Suppose for a contradiction that some $x\in K_n$ with $n \geq 1$ has two good children
$y_1$ and $y_2$. Since they are different, there is an element $i$ of $J_{n+1}$ separating them, and a bridge $e$ in $I_{n+1}$ of which $i$ is an endvertex. Since $i \not \in J_n$, $e \not \in I_{n}$ and so the other endvertex $j$ of $e$ must lie in $p(x) = p(p(y_1)) = p(p(y_2))$.
Now the two endvertices of $e$  have to be in different components of $\overline{p(p(y_1))\sm y_1}$
or $\overline{p(p(y_2))\sm y_2}$. Hence $y_1$ and $y_2$ cannot both be good at the same time, a contradiction.
\end{proof}

\begin{lem}\label{all_but_fin_good}
Let $(k_n)$ be a ray whose partition class includes an edge.
Then all but finitely many nodes on it are good.
\end{lem}

\begin{proof}
Let $e$ be the edge in the partition class of this ray.
Let $f$ be any edge of $C\sm k_0$.

Suppose for a contradiction that there is an infinite set $N$ of natural numbers such that
$k_n$ is not good for any $n\in N$. Let $N'$ be an infinite subset of $N$ that does not contain 0, 1 or any pair of consecutive natural numbers.
For each $n \in N'$, pick a bridge $e_n$ in $I_n$ with endvertices 
in both components of $\overline{p(p(k_n))\sm k_n}$, which is possible since $k_n$ is not good. The endvertices of $e_n$ are in $J_n$ but not $J_{n-2}$ and so we cannot find $m\neq n \in N'$ such that $e_m$ and $e_n$ share an endvertex. 
Applying \autoref{cir_eli_single} with $r_1=e$, $r_2=f$ and $W = \{e_n| n \in \Nbb\}$ yields that $G'$ does not induce a matroid, a contradiction. This completes the proof.
\end{proof}

\begin{proof}[Proof of \autoref{countcir}.]
For each edge of $C$ there is a unique ray whose partition class contains that edge. 
By \autoref{all_but_fin_good}, we can find a first node on that ray such that it and all successive nodes are good. This gives a map from the edges of $C$ to the nodes of the partition tree. By \autoref{good_son} and \autoref{part_leq_1}, this map is injective. By \autoref{few_sons} the partition tree has only countably many nodes.
\end{proof}

Having proved \autoref{countcir}, we conclude with a slight strengthening.

\begin{cor}\label{p-arcs_count2}
Let $G$ be a graph-like space inducing a $3$-connected matroid $M$ such that any vertex of $G$ lies on 
a pseudo-circle. Then every pseudo-arc of $G$ is countable.
\end{cor}

\begin{proof}
Let $Q$ be some pseudo-arc in $G$, and let $x$ and $y$ be its endvertices.
By \autoref{cor8_2}, there is a pseudo-circle
containing the vertices $x$ and $y$. By \autoref{countcir}, this pseudo-circle has only countably many edges, and so it includes an $x$-$y$-pseudo-arc $P$ with only countably many edges. 

We shall now compare the pseudo-arc $P$ with $Q$ in order to show that $Q$ also has only countably many edges.
For this, we extend the edge set $E(Q)$ of $Q$ to a base $s$ of $M$.
Now we define a bipartite graph whose left bipartition class is $E(Q)$, and whose right class is $E(P)$. We join $q\in E(Q)$ to $p\in E(P)$ if the fundamental cocircuit of $q$ with respect to $s$ contains $p$. For any $p \in E(P) \cap s$, $p$ can have at most one neighbour in $Q$. For any $p \in E(P) \setminus s$, by \autoref{fdt} any neighbour of $p$ must lie in the fundamental circuit of $p$ in $s$. Since this circuit is countable by \autoref{countcir},
each $p$ has only countably many neighbours.

Next, we show that each $q\in E(Q)$ has at least one neighbour $p\in E(P)$. 
Let $b_q$ be the fundamental cocircuit of $q$.
Let $v$ and $w$ be the two endvertices of $q$. Then
$v$ and $w$ are on different sides of $b_q$, and since $Qs$ and $tQ$ both avoid
$b_q$, the vertices $v$ and $w$ are on different sides of $b_q$.
Hence the $x$-$y$-pseudo-arc $P$ has to meet $b_q$ in some edge $p$.

Having shown that each $q\in E(Q)$ has at least one neighbour in $E(P)$,
we are now in a position to show that $E(Q)$ is countable.
For this we write $E(Q)$ as a countable union of countable sets, namely the neighbourhoods of the $p \in E(P)$. This completes the proof.
\end{proof}

\section{Planar graph-like spaces}\label{sec10}

A nice consequence of \autoref{countcir} is the following.

\begin{cor}\label{planar}
Let $M$ be a tame $3$-connected matroid such that all finite minors are planar.
Then $E(M)$ is at most countable.
\end{cor}

\begin{proof}
Let $e$ be some edge.
By \autoref{swseq}, there is a switching sequence from $e$ to any other edge.
Hence it suffices to show that there are only countably many different switching sequences starting at $e$. We show by induction that there are only countably many switching sequences of length $n$ for each $n$. The case $n=1$ is obvious.
The first $n-1$ elements of a switching sequence of length $n$ 
form a switching sequence of length {$n-1$}. On the other hand, there are only countably many ways to extend a given switching sequence of length $n-1$ to one of length $n$ since all  circuits and cocircuits of $M$ are countable by \autoref{countcir}.
Hence there are only countably many  switching sequences of length $n$.
This completes the proof.
\end{proof}

This raises the question how to embed the graph-like space constructed from a tame matroid all of whose finite minors are planar in the plane.
However, we shall construct such a matroid that does not seem to be embeddable in this sense the plane.
Let $N$ be the matroid whose circuits are the edge sets of topological circles in the topological space depicted in \autoref{not-planar}. We omit the proof that this gives a matroid - it can be found in \cite{C:algebraic_psi-matroids}. However, much of the complication of this matroid was introduced to make it 3-connected, and if we do not require 3-connectedness then it is easy to construct other simpler examples sharing the essential property of this matroid: it is tame and all finite minors are planar, but the topology of the graph-like space it induces has no countable basis of neighbourhoods for the vertex at the apex, so it cannot be embedded into the plane.

   \begin{figure} [htpb]   
\begin{center}
   	  \includegraphics[width=10cm]{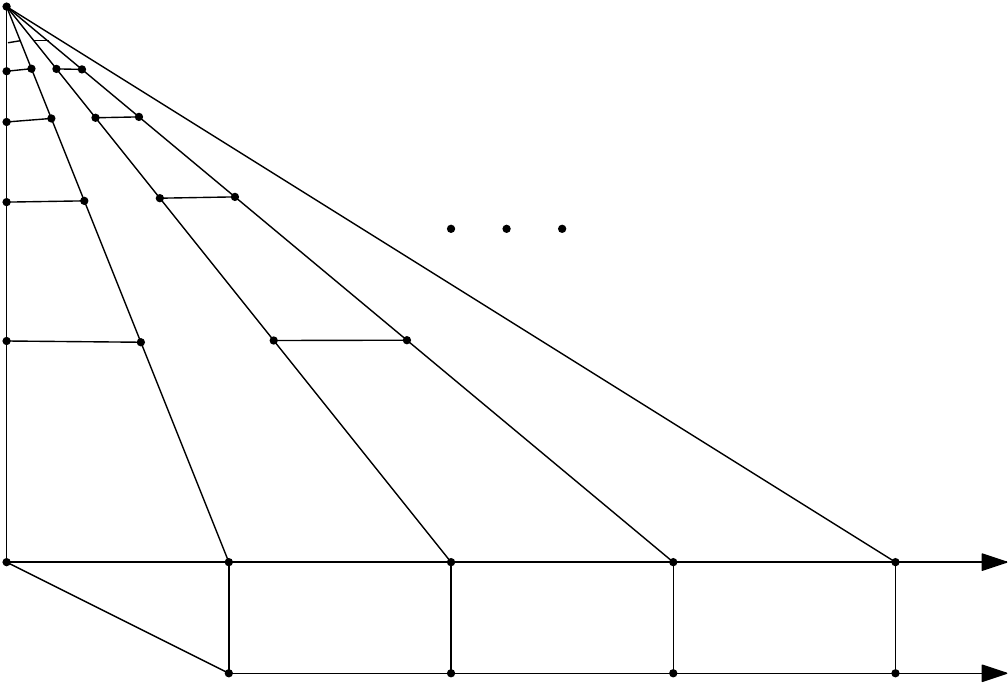}
   	  \caption{The matroid $N$.}
   	  \label{not-planar}
\end{center}
   \end{figure}

\bibliographystyle{plain}
\bibliography{literatur-1}

\end{document}